\documentclass[12pt,a4paper]{amsart}
\usepackage{amsmath}%
\usepackage{amssymb}%
\usepackage{amscd}%
\usepackage{amsthm}%
\usepackage{amsfonts}%

\usepackage{enumerate}%
\usepackage{tikz}%
\usepackage{tikz-cd}%
\usepackage{verbatim}%
\usepackage{xspace}%
\usepackage{tensor}%
\usepackage{graphicx}%
\usepackage{appendix}%
\usepackage{color}%
\usepackage{epsfig}%
\usepackage{latexsym}%
\usepackage[hidelinks]{hyperref}%
\usepackage{fullpage}%
\usepackage{pdfsync}%

\numberwithin{equation}{section}%

\theoremstyle{plain}%
\newtheorem{theorem}{Theorem}[section]%
\newtheorem{proposition}[theorem]{Proposition}%
\newtheorem{corollary}[theorem]{Corollary}%
\newtheorem{lemma}[theorem]{Lemma}%
\newtheorem{hypothesis}[theorem]{Hypothesis}%

\theoremstyle{definition}%
\newtheorem{definition}[theorem]{Definition}%
\newtheorem{example}[theorem]{Example}%

\theoremstyle{remark}%
\newtheorem*{remark}{Remark}%
\newcommand\cut[1]{}%
\newcommand\further[1]{}%
\begin{document}
	\title[Infinite Moments of Class Groups]{Infinite Moments of Class Groups for Solvable Fields \\ with a Normal Abelian Subgroup}%
	\date{}%
	\author{Weitong Wang}%
	\address{College of Mathematical Sciences, Shaw Hall of Harbin Engineering University, No. 145, Nantong Street, Nangang District, Harbin, Heilongjiang Province 150001 China}%
	\email{weitongwang@hrbeu.edu.cn}%
\begin{abstract}
	We apply the class field theory and Minkowski bound to obtain an upper bound estimate for the number of solutions to the restricted ramifications when the Galois group is solvable.
	Together with suitable conditions on the solvable group and the ordering of number fields, we could prove an upper bound on specific field-counting problems, hence the infinite moment of the class groups.
	In particular, for non-Galois cubic fields ordered by the product of ramified primes, we could show that the $\mathbb{Z}/3\mathbb{Z}$-moment is infinite with the results on the $\mathbb{Z}/3\mathbb{Z}$-moment of quadratic number fields and the field-counting on cubic fields ordered by the generalized discriminant.
\end{abstract}	
	\maketitle
	\tableofcontents
	\section{Introduction}\label{section: intro}
	In this paper, we are mainly interested in the distribution of class groups of number fields.
	Let us use an example to explain the notion briefly.
	Let $\mathcal{C}$ be the set of quadratic number fields ordered by the absolute discriminant $d$.
	Define
	\begin{equation*}
		N_{\mathcal{C},d}(X)=\#\{K\in\mathcal{C}\mid d_K<X\},
	\end{equation*}
	which is the function that counts quadratic number fields ordered by discriminant.
	Then we can define the notations of probability and moments for class groups.
	Let $A$ be a finite abelian group, and $p$ be a rational prime.
	Define the $p$-rank of $A$, denoted by $\operatorname{rk}_pA$, as the largest number $r$ so that there exists some injective group homomorphism $(\mathbb{Z}/p\mathbb{Z})^r\to A$.
	Since $A$ is finite abelian, we also have
	\begin{equation*}
		\operatorname{rk}_{p}A=\dim_{\mathbb{F}_{p}}A/pA.
	\end{equation*}
	For each non-negative integer $r$, define
	\begin{equation*}
		\begin{aligned}
			\mathbb{P}_{\mathcal{C},d}(\operatorname{Cl}_K\cong A):=
			&\lim_{X\to\infty}
			\frac{
				\#\{K\in\mathcal{C}\mid d_K<X\text{ and }\operatorname{Cl}_K\cong A\}
			}
			{
				N_{\mathcal{C},d}(X)
			}
			\\
			\mathbb{P}_{\mathcal{C},d}(\operatorname{rk}_p\operatorname{Cl}_K\leq r):=
			&\lim_{X\to\infty}
			\frac{
				\#\{K\in\mathcal{C}\mid d_K<X\text{ and }\operatorname{rk}_p\operatorname{Cl}_K\leq r\}
			}
			{
				N_{\mathcal{C},d}(X)
			},
		\end{aligned}
	\end{equation*}
	if the limit exists,
	and call it the probability of $\operatorname{Cl}_K\cong A$, resp. $\operatorname{rk}_p\operatorname{Cl}_K\leq r$.
	Define the $A$-moment of $\operatorname{Cl}_K$ to be
	\begin{equation*}
		\mathbb{E}_{\mathcal{C},d}(\lvert\operatorname{Hom}(\operatorname{Cl}_K,A)\rvert):=
		\lim_{X\to\infty}
		\frac{
			\sum_{
				\substack{
					K\in\mathcal{C}\\
					d_K<X
				}
			}
			\lvert\operatorname{Hom}(\operatorname{Cl}_{K},A)\rvert
		}{
			N_{\mathcal{C},d}(X)
		},
	\end{equation*}
	provided that the limit exists.
	When $p$ is an odd prime, Cohen and Lenstra~\cite{CL84} gives the prediction for the probability distribution of $\operatorname{Cl}_K\otimes\mathbb{Z}_p\cong A$, where $A$ is a finite abelian $p$-group.
	When the set of fields $\mathcal{C}$ is generalized to the set of Galois $G$-fields, where $G$ is a finite group, Cohen and Martinet~\cite{CM90} gives the corresponding generalization of Cohen-Lenstra Heuristics.
	The method of Cohen and Martinet could be applied to non-Galois cases and obtain the corresponding predictions.
	See the author and Wood~\cite{wang2021moments}.
	Though this area is widely open, there are some proven results.
	A famous one is obtained by Davenport and Heilbronn~\cite{Davenport1971Cubic}.
	In the context of distribution of class groups, we can translate their result into the following: 
	for quadratic number fields, the $\mathbb{Z}/3\mathbb{Z}$-moment of class groups is exactly what is predicted by Cohen-Lenstra-Martinet Heuristics.
	Another result is the work of Alex Smith~\cite{smith2026distributionI,smith2026distributionII}, which proves that the distribution of $\operatorname{Cl}_K\otimes\mathbb{Z}_2$ satisfies the Gerth's Conjecture~\cite{Gerth1987} when $K$ runs over quadratic number fields.
	
	We discuss in detail the case where $p=2$ and $K$ quadratic.
	If we apply Genus Theory (see Ishida~\cite{ishida1976genus} for example) to quadratic number fields, then
	\begin{equation*}
		\omega(d_K)-1\leq\operatorname{rk}_2\operatorname{Cl}_K\leq\omega(d_K),
	\end{equation*}
	where $\omega(n)$ counts all the distinct prime factors of an integer $n$.
	This implies that for each non-negative integer $r$, we have that
	\begin{equation*}
		\mathbb{P}_{\mathcal{C},d}(\operatorname{rk}_2\operatorname{Cl}_K\leq r)=0,
	\end{equation*}
	and
	\begin{equation*}
		\mathbb{E}_{\mathcal{C},d}(\lvert\operatorname{Hom}(\operatorname{Cl}_K,\mathbb{Z}/2\mathbb{Z})\rvert)=+\infty.
	\end{equation*}
	We can call this phenomenon ``zero-probability'', resp. ``infinite moment'' in short.
	This means that the distribution of $\operatorname{Cl}_K\otimes\mathbb{Z}_2$ is qualitatively different from other Sylow $p$-subgroups of $\operatorname{Cl}_K$.
	In particular, the original Cohen-Lenstra Heuristics \emph{cannot} be applied to this case (there is literally no prediction from the heuristics).
	This is a motivation for Gerth's Conjecture.
	And we will follow this phenomenon in this paper.
	
	Let us introduce some notations so that we can make statements clearly.
	\begin{definition}\label{def:set of primes}
		Fix a number field $K$.
		Denote by $\mathcal{P}_K$ the set of all primes of $K$, including the ones at infinity.
		When $K=\mathbb{Q}$, let $\mathcal{P}:=\mathcal{P}_{\mathbb{Q}}$ be the set of all rational primes and the infinity.
		Given a set $\mathfrak{R}$ of ideals of $K$, define $\mathcal{P}_{\mathfrak{R}}$ as the set of all primes in $\mathfrak{R}$.
	\end{definition} 
	An example of $\mathfrak{R}$ is the set given by an ideal class of $\operatorname{Cl}_K$.
	Given a field $F$, we say that $A$ is an $F$-algebra if $A$ is a (not necessarily commutative) ring with a fixed injective map $F\to A$ such that $F$, as a subalgebra, is included in the center of $A$.	
	To abstractly define a field extension up to isomorphism so that we could count them, let us introduce the following notation.
	\begin{definition}\label{def: Galois algebras}		
		Let $G$ be a finite group, and $F$ be a fixed field.
		We say that $(A/F,\varphi_A)$ is a Galois $(G,F)$-algebra if $A$ is an {\'e}tale $F$-algebra of degree $\lvert G\rvert$ with a $G$-action defined by $\varphi_A:G\hookrightarrow\operatorname{Aut}_F(A)$ such that $A^G=F$.
		A morphism between two Galois $(G,F)$-algebras $A$ and $B$ is a morphism of $F$-algebras $f:A\to B$ such that for each $g\in G$, the following diagram commutes
		\begin{equation*}
			\begin{tikzcd}
				A\arrow[r,"f"]&B\arrow[d,"\varphi_B(g)"]\\
				A\arrow[u,leftarrow,"\varphi_A(g)"]&B\arrow[l,leftarrow,"f"]
			\end{tikzcd}
		\end{equation*}
		If a Galois $(G,F)$-algebra $A$ is a field itself, then we simply call it a Galois $G$-field extension over $F$.
	\end{definition}
	If $(A/F,\varphi_A)$ is a Galois $G$-field extension, then $\varphi_A:G\to\operatorname{Gal}(A/F)$ is a group isomorphism, for $[A:F]=\lvert G\rvert$ and $A^G=F$.
	To some sense, the notion of Galois algebras is a generalization of the Galois field extensions with a fixed $G$-action.
	We will show the following properties of Galois algebras.
	\begin{definition}\label{def: induced algebra}
		Let $G$ be a finite group with a fixed subgroup $H$, and $F$ be a fixed field.
		Let $A$ be an $F$-algebra with an $H$-action ($H\to\operatorname{Aut}_{F}(A)$).
		Define $\operatorname{Ind}^{G}_{H}A$ to be the $G$-algebra as follows.
		As a set, define
		\begin{equation*}
			\operatorname{Ind}^{G}_{H}A:=\{\tau\in\operatorname{Map}(G,A)\mid\forall h\in H,\forall g\in G,\, \tau(hg)=h\tau(g)\}.
		\end{equation*}
		The addition and the multiplication are just defined pointwise:
		\begin{equation*}
			(\tau_{1}+\tau_{2})(g)=\tau_{1}(g)+\tau_{2}(g)\quad\text{and}\quad (\tau_{1}\cdot\tau_{2})(g)=\tau_{1}(g)\cdot\tau_{2}(g).
		\end{equation*}
		For each $\lambda\in F$, we have $(\lambda\tau)(g)=\lambda\tau(g)$.
		For each $x,g\in G$, we have $(x\cdot\tau)(g)=\tau(x^{-1}g)$.
		In particular, as a $G$-module, we have
		\begin{equation*}
			\operatorname{Ind}^{G}_{H}A\cong F[G]\otimes_{F[H]}A.
		\end{equation*}
		And $\operatorname{Ind}^{G}_{H}$ is a covariant functor from the category of $(H,F)$-algebras to the category of $(G,F)$-algebras such that for each $(G,F)$-algebra $B$ which is also viewed as an $H$-algebra (via the restriction), we have
		\begin{equation*}
			\operatorname{Hom}_{G}(\operatorname{Ind}^{G}_{H}A,B)\cong\operatorname{Hom}_{H}(A,B)
		\end{equation*}
		where the isomorphism is given by the natural transformation:
		the morphism of $(H,F)$-algebras $\varphi:A\to B$ is sent to the morphism of $(G,F)$-algebras
		\begin{equation*}
			\Phi:\operatorname{Ind}^{G}_{H}A\to B,\, \tau\mapsto\sum_{gH\in G/H}g\cdot\varphi(\tau(g^{-1})).
		\end{equation*}
	\end{definition}
	The basic structure of the Galois algebras is given by the following.
	Roughly speaking, every Galois $(G,F)$-algebra $A$ is of the form $A\cong\operatorname{Ind}^{G}_{H}K$ as Galois $(G,F)$-algebras, where $H$ is a subgroup of $G$, and $K/F$ is a Galois $H$-extension.
	\begin{theorem}\label{thm: structure of Galois algebras}
		Let $G$ be a finite group, $F$ be a field such that $\operatorname{char}(F)\nmid\lvert G\rvert$, and $(A/F,\varphi_A)$ be a Galois $(G,F)$-algebra.
		For each primitive (central) idempotent $e\in A$, let $G_e:=\operatorname{Stab}_G(e)$ be its stabilizer.
		The component $eA$ is a Galois $G_e$-extension over $F$ with the group isomorphism $\varphi_{eA}:G_e\to\operatorname{Gal}(eA/F)$ defined by $g\mapsto\varphi_A(g)\vert_{eA}$.
		In particular, $A\cong\operatorname{Ind}^G_{G_e} eA$ as Galois $(G,F)$-algebras, where $\operatorname{Ind}$ means the induced algebra.
	\end{theorem}
	See Section~\ref{sec: specifications} for the discussions on the Galois algebras.
	The theorem immediately implies that for each primitive idempotent $e$ and $e'$ of $A$, the fields $eA$ and $e'A$ are isomorphic.
	Actually we will show that they could be identified as Galois $F$-algebras by a conjugate from $G$ (clearly $G_e$ and $G_{e'}$ are conjugate in $G$).
	Therefore, we could define the ramification of $A$ over $F$ when $F$ is a local field, generalizing the case when $A$ is a field.
	\begin{definition}\label{def: ramification of Galois algebra}
		Let $G$ be a finite group, $F/\mathbb{Q}_p$ be a local field, and $A/F$ be a Galois $(G,F)$-algebra.
		Assume that $e\in A$ is a primitive idempotent and $\mathfrak{p}$ is the valuation of $F$.
		Then the inertia subgroup $I_{\mathfrak{p}}\subseteq G_e\subseteq G$ is the inertia subgroup of $A/F$ up to conjugation.
		When $\mathfrak{p}\nmid\lvert G\rvert$, the inertia subgroup $I_{\mathfrak{p}}$ is cyclic and generated by some element $y_{\mathfrak{p}}\in G$ up to conjugation and invertible powering.		
	\end{definition}
	Let $K/k$ be a Galois number field extension $K/k$ with $G(K/k)\cong G$, the definition of specifications at $\mathfrak{p}$ reflects the global-local principle, that is, $K_{\mathfrak{p}}=K\otimes k_{\mathfrak{p}}=\operatorname{Ind}^G_{G_{\mathfrak{P}}}K_{\mathfrak{P}}$, where $\mathfrak{P}$ is a prime of $K$ above $\mathfrak{p}$.
	So, we can define local specifications by Galois algebras.
	\begin{definition}\label{def: local specification}
		Let $k$ be a number field and $G$ be a finite group.
		\begin{enumerate}
			\item For each prime $\mathfrak{p}$ of $k$, we say that $\Sigma_{\mathfrak{p}}$ is a $G$-specification at $\mathfrak{p}$ if it is a Galois $(G,k_{\mathfrak{p}})$-algebra.
			\item Let $S$ be a (possibly infinite) set of primes of $k$ including the infinite ones.
			We say that $\Sigma=(\Sigma_{\mathfrak{p}})_{\mathfrak{p}\in S}$ is a $G$-specification at $S$ if it is a product of $G$-specifications at $\mathfrak{p}\in S$ such that $\Sigma_{\mathfrak{p}}/k_{\mathfrak{p}}$ is unramified for almost all $\mathfrak{p}\in S$.
			\item When $S$ is taken to be the set of all primes of $k$, we simply say that a specification $\Sigma$ at $S$ is a $G$-local specification.
			\item Let $\Sigma$ be a $G$-specification at $S$, and $K/k$ be a Galois $G$-extension.
			We say that $K$ is a solution to $(G,k,\Sigma)$, denoted by $K\sim\Sigma$, if $K$ is unramified outside $S$ and for each $\mathfrak{p}\in S$ we have $K_{\mathfrak{p}}=K\otimes k_{\mathfrak{p}}\cong\Sigma_{\mathfrak{p}}$ as Galois $(G,k_{\mathfrak{p}})$-algebras.
			If there exists at least one solution $K/k$ to $(G,k,\Sigma)$, then we say that $\Sigma$ is admissible.
		\end{enumerate}
	\end{definition}	
	Using the notation of local specifications, we can generalize the notion of the (absolute) discriminant.
	Let $K/k$ be a Galois $G$-extension of number fields, then a prime $\mathfrak{p}\nmid\lvert G\rvert$ of $k$ admits an inertia generator $y_{\mathfrak{p}}$ (up to conjugate) as an element of $G$.
	The ramification of $\mathfrak{p}$ in $K/k$ will determine its exponent in the relative discriminant $\mathfrak{d}_{K/k}$.
	Similar description also works for the conductor (when $G$ is abelian) and the Artin conductor in general.
	Let us follow this idea and define the generalized discriminant as follows.
	See also Wood~\cite[Section 2]{wood2010probabilities}.	
	We first introduce the notation of the set of fields.
	\begin{definition}\label{def:set of fields}
		Fix a number field $k$, and a transitive permutation group $G\subseteq S_n$ with $d=\lvert G\rvert$.
		For a field extension $K/k$, let $\hat{K}/k$ be its Galois closure.
		We say that $K/k$ is a $(d,G)$-extension if $(\hat{K}/k,\varphi)$ is a Galois $G$-extension and $K=\hat{K}^{G_1}$ where $G_1=\operatorname{Stab}_G(1)$.
		Two $G$-extensions $K_1/k$ and $K_2/k$ are isomorphic if $(\hat{K}_1/k,\varphi_1)\cong(\hat{K}_2/k,\varphi_2)$ as Galois $(G,k)$-fields.
		Define ${\mathcal{C}}(G,k)$ as the set of $G$-extensions $(K,\psi)$ up to isomorphism.
		If the base field $k=\mathbb{Q}$, then we just omit it and write ${\mathcal{C}}(G):={\mathcal{C}}(G,\mathbb{Q})$.
	\end{definition}
	\begin{remark}
		\begin{enumerate}
			\item There are alternative ways to define the set of fields.
			See Wood~\cite[p.292]{directions2016} for their differences and connections.
			\item Note that if $G$ is abelian, and we view $G$ as a transitive permutation group by its operation on itself, then $\mathcal{C}(G,k)$ simply means the set of abelian $G$-extensions over $k$.
			\item If $G$ is abelian, there is a one-to-one correspondence between the surjective continuous homomorphisms $\operatorname{C}_k\to G$, where $\operatorname{C}_k$ is the id{\`e}les class group, and the set of $G$-extensions $K/k$.
			By Class Field Theory, the open normal subgroups of $\operatorname{C}_k$ corresponds to the abelian field extensions.
			And in general, there are multiple surjective maps $\operatorname{C}_k\to G$ with the same kernel $N_{K/k}\operatorname{C}_K$.			
			But they define different actions of $G$ on the extension $K/k$, that is, they specify different isomorphisms $\operatorname{Gal}(K/k)\cong G$.
			So, we know that different surjective maps $\operatorname{C}_k\to G$ with the same kernel corresponds to different $G$-extensions $K/k$, though the underlying fields $K$ are isomorphic.
		\end{enumerate}
	\end{remark}
	Let us give the notation of counting number fields.
	\begin{definition}\label{def: counting number fields}
		Let $S$ be a countable set with a function $C:S\to\mathbb{R}_{\geq0}$ such that for each $X>0$ the set $\{a\in S\mid C(a)<X\}$ is finite.
		Define
		\begin{equation*}
			N_{S,C}(X):=\#\{a\in S\mid C(a)<X\}.
		\end{equation*}
	\end{definition}
	For now, we have seen that the (absolute) discriminant could work as a counting function.
	However, in some cases, ordering fields by discriminant will contradict what is predicted by the heuristics.
	See Cohen and Martinet~\cite{Cohen1994HeuristicsOC}, Bartel and Lenstra~\cite{bartel2020class} for example.
	See also Wood~\cite{wood2010probabilities} for some discussions on different orderings from a field-counting point of view.
	The choice of the counting function may affect the result of field-counting in a nontrivial way.
	But we are not going to discuss it in detail here.
	Let us give the definition of a counting function based on local specifications.
	\begin{definition}\label{def: generalized discriminant}
		Let $G$ be a transitive permutation group with $\operatorname{Stab}_{G}(1)$ being trivial, and $\mathcal{C}$ be a subset of $\mathcal{C}(G,K)$ with $K$ a fixed number field.
		\begin{enumerate}
			\item We say that $C:\mathcal{C}\to\mathbb{R}_+$ is a \emph{counting function} if for each $\mathfrak{p}\in\mathcal{P}_K$ there exists a function $C_{\mathfrak{p}}:\{$ specifications at $\mathfrak{p}\}\to\mathbb{R}_{+}$ such that
			\begin{equation*}
				C(L/K)=\prod_{\mathfrak{p}\in\mathcal{P}_K}C(L_{\mathfrak{p}}/K_{\mathfrak{p}})
				\quad\text{and}\quad
				\forall X\in\mathbb{R}_+,\,\#\{L/K\in\mathcal{C}\mid C(L/K)<X\}<\infty
			\end{equation*}
			where $L_{\mathfrak{p}}=L\otimes_K K_{\mathfrak{p}}$.
			\item For each $g_{1},g_{2}\in G$, we say that they are equivalent under conjugation and invertible powering, denoted by $g_{1}\sim g_{2}$, if there exists some integers $a,b\in\mathbb{Z}$ and $h\in G$ such that
			\begin{equation*}
				g_{1}=hg_{2}^{a}h^{-1}\quad\text{and}\quad g_{2}=h^{-1}g_{1}^{b}h.
			\end{equation*}
			\item We call a function $c_G:G\to\mathbb{R}_{\geq0}$ as a \emph{weight} of $G$ if it satisfies the following two conditions simultaneously:
			\begin{enumerate}[(i)]
				\item $c_G(g)=0$ if and only if $g$ is the identity of $G$;
				\item for each $g,h\in G$, if $g\sim h$, then $c_{G}(g)=c_{G}(h)$.
			\end{enumerate} 
			\item For each prime $\mathfrak{p}\nmid\lvert G\rvert\infty$ of $K$, and for each $G$-specification $\Sigma_{\mathfrak{p}}$ at $\mathfrak{p}$, let $y_{\mathfrak{p}}$ be the inertia generator defined up to invertible powering and conjugation.
			Define the local discriminant $C_{\mathfrak{p}}:\{G-$specifications at $\mathfrak{p}\}\to\mathbb{R}_{+}$ with respect to the weight $c_G$ by the following:
			\begin{equation*}
				C_{\mathfrak{p}}(\Sigma_{\mathfrak{p}})=\mathfrak{Np}^{c_{G}(y_{\mathfrak{p}})},
			\end{equation*}
			where $\mathfrak{Np}$ is the absolute norm for finite prime $\mathfrak{p}$.
			For each $\mathfrak{p}\mid\lvert G\rvert\infty$ of $K$, a local generalized discriminant $C_{\mathfrak{p}}$ is any map $C_{\mathfrak{p}}:\{G-$specifications at $\mathfrak{p}\}\to\mathbb{R}_{+}$.
			\item We define the \emph{generalized discriminant} $C:\mathcal{C}\to\mathbb{R}_+$ with respect to $c_G$ and $\{C_{\mathfrak{p}}\}_{\mathfrak{p}\mid\lvert G\rvert\infty}$ as a counting function by the formula
			\begin{equation*}
				C(L/K):=\prod_{\mathfrak{p}}C_{\mathfrak{p}}(L_{\mathfrak{p}}/K_{\mathfrak{p}})
			\end{equation*}			 
			\item We sometimes omit the data for $\mathfrak{p}\mid\lvert G\rvert\infty$ and just say $C$ is a generalized discriminant with respect to $c_G$.
			Moreover, we say that two generalized discriminants $C_1$ and $C_2$ are equivalent if there exists some positive number $a$, for each $\mathfrak{p}\nmid\lvert G\rvert\infty$ and for each $G$-specification $\Sigma_{\mathfrak{p}}$ at $\mathfrak{p}$ we have
			\begin{equation*}
				C_{1,\mathfrak{p}}(\Sigma_{\mathfrak{p}})=C_{2,\mathfrak{p}}(\Sigma_{\mathfrak{p}})^{a}.
			\end{equation*}
			In other words, for a Galois $G$-extension $L/K$, the generalized discriminants $C_1$ is a power of $C_2$ up to the wildly ramified primes.
		\end{enumerate}		
	\end{definition}
	Note that when $\mathfrak{p}\nmid\lvert G\rvert\infty$, the value $c_G(y_{\mathfrak{p}})$ is independent of the choice of the inertia generator $y_{\mathfrak{p}}$, which is well-defined up to conjugation and invertible powering.
	We will generalize this definition to count morphisms later.
	Let us define the following notations to describe the distribution of class groups.
	\begin{definition}\label{def: p-rank and moment}
		Let $\mathcal{C}$ be a set of number fields with a counting function $C$, and $p$ be a finite rational prime.
		\begin{enumerate}
			\item Define
			\begin{equation*}
				\mathbb{P}_{\mathcal{C},C}(\operatorname{rk}_p\operatorname{Cl}_K\leq r):=\lim_{X\to\infty}\frac{\#\{K\in\mathcal{C}\mid C(K)<X\text{ and }\operatorname{rk}_p\operatorname{Cl}_K\leq r\}}{N_{\mathcal{C},C}(X)}.
			\end{equation*}
			\item For a fixed finite abelian group $B$, define the $B$-moment of $\operatorname{Cl}_K$ where $K$ runs over fields in $\mathcal{C}$ ordered by $C$ to be
			\begin{equation*}
				\mathbb{E}_{\mathcal{C},C}(\lvert\operatorname{Hom}(\operatorname{Cl}_K,B)\rvert):=
				\lim_{X\to\infty}
				\frac{
					\sum_{
						\substack{
							K\in\mathcal{C}\\
							C(K)<X
							}
						}
						\lvert\operatorname{Hom}(\operatorname{Cl}_K,B)\rvert
					}{
					N_{\mathcal{C},C}(X)
					},
			\end{equation*}
			if the limit exists.
		\end{enumerate}
	\end{definition}
	The following result is on the distribution of class groups.
	\begin{theorem}\label{thm: solvable group with abelian normal subgroup S1}
		Let $G$ be a transitive permutation group with $\operatorname{Stab}_G(1)$ trivial, and $N$ be an abelian normal subgroup such that $\gcd(\lvert N\rvert,\lvert G/N\rvert)=1$.
		Define
		\begin{equation*}
			\mathcal{C}:=\{L\in\mathcal{C}(G)\mid \mu(L)=\mu(\mathbb{Q})\},
		\end{equation*}
		where $\mu(L)$ is the group of roots of unity.
		Fix a rational prime $p\mid\lvert N\rvert$.
		Let $c_{G}$ be a weight such that $\displaystyle{m:=\min_{e_{G}\neq g\in G}\{c_{G}(g)\}}$ and $c_{G}(g)=m$ only if $g\in N$ and $r_g\equiv0\bmod{p}$, where $r_g$ is the order of $g$ in the group.
		Define $\Omega:=\{g\in G\mid c_G(g)=m\}$, and $C$ to be a generalized discriminant associated to $c_{G}$.
		Let $H$ be a complement of $N$ in $G$.
		If $H$ is solvable, then for each $r\in\mathbb{Z}_{\geq0}$, we have
		\begin{equation*}
			\mathbb{P}_{\mathcal{C},C}(\operatorname{rk}_p\operatorname{Cl}_K\leq r)=0,
		\end{equation*}
		where $\operatorname{rk}_p$ is the $p$-rank of finite abelian groups which could be defined as $\operatorname{rk}_p A:=\operatorname{dim}_{\mathbb{F}_p}A/pA$.
		In addition, we have
		\begin{equation*}
			\mathbb{E}_{\mathcal{C},C}(\lvert\operatorname{Hom}(\operatorname{Cl}_K,\mathbb{Z}/p\mathbb{Z})\rvert)=+\infty.
		\end{equation*}
	\end{theorem}
	See Section~\ref{section: Solvable extension with a normal abelian subgroup} for its proof. %
	This result focuses on the case when the Galois group $G$ is solvable and includes an abelian normal subgroup.
	With a suitable generalized discriminant, we have the zero-probability distribution for the $p$-primary part of the class groups and the infinite $\mathbb{Z}/p\mathbb{Z}$-moment.
	When we know more about the moment of class groups, we could prove some statistical results when the generalized discriminant is the product of ramified primes.
	To be precise, we have the following.
	\begin{theorem}\label{thm: cubic fields S1}
		Let $S_{3}$ be the symmetric group acting on $3$ elements, and $\mathcal{C}:=\mathcal{C}(S_3,\mathbb{Q})$.
		For each $K\in\mathcal{C}$, define $C(K):=df$ where $d$ is the (absolute) discriminant of the associated quadratic number field and $f$ is the product of totally ramified primes.
		For each $r\in\mathbb{Z}_{\geq0}$, we have
		\begin{equation*}
			\mathbb{P}_{\mathcal{C},C}(\operatorname{rk}_3\operatorname{Cl}_K\leq r)=0
			\quad\text{and}\quad
			\mathbb{E}_{\mathcal{C},C}(\lvert\operatorname{Hom}(\operatorname{Cl}_{K},\mathbb{Z}/3\mathbb{Z})\rvert)=+\infty.
		\end{equation*}
	\end{theorem}
	Compared to the Theorem~\ref{thm: solvable group with abelian normal subgroup S1}, the main difference is that the generalized discriminant is changed to the product of ramified primes.
	And this condition makes the problem more subtle, hence requiring more information from the moments of class groups and the field-counting.
	For the non-Galois cubic fields, we take the advantage of the famous result on the $\mathbb{Z}/3\mathbb{Z}$-moment of the class groups of the quadratic fields by Davenport and Heilbronn~\cite{Davenport1971Cubic}, as well as the result by Shankar and Thorne~\cite{shankar2024asymptoticscubicfieldsordered} of counting cubic fields ordered by the radical of the discriminant.
	See Section~\ref{section: product of ramified primes} for the proof.
	
	\section{Basic notations}\label{section:notation}
	In this section we introduce some of the notations that will be used in the paper.
	We use some standard notations coming from analytic number theory.
	For example, write a complex number as $s=\sigma+it$.
	Denote the Euler's phi function by $\phi(n)$.
	Let $\omega(n)$ count the number of distinct prime divisors of $n$ and so on.
	
	We also follow the notations of inequalities with unspecified constants from Iwaniec and Kowalski~\cite[Introduction, p.7]{iwaniec2004analytic}.
	Let us just write down the ones that are important for us.
	Let $X$ be some space (usually some region of $\mathbb{C}$ in our paper), and let $f,g$ be two complex functions defined on $X$.
	Then $f(x)\ll g(x)$ for $x\in X$ means that $\lvert f(x)\rvert\leq C\lvert g(x)\rvert$ for some constant $C\geq0$.
	Any value of $C$ for which this holds is called an implied constant.
	We use $f(x)\asymp g(x)$ for $x\in X$ if $f(x)\ll g(x)$ and $g(x)\ll f(x)$ both hold with possibly different implied constants.
	We say that $f=o(g)$ as $x\to x_0$ if for any $\epsilon>0$ there exists some (unspecified) neighbourhood $U_\epsilon$ of $x_0$ such that $\lvert f(x)\rvert\leq\epsilon\lvert g(x)\rvert$ for $x\in U_\epsilon$.
	Finally, $f\sim g$ as $x\to x_0$ if we can write $f=g+o(g)$.
	
	Throughout the paper, without further explanation, the notation $\operatorname{Hom}$ always refers to the continuous maps.
	For example, let $G$ be a finite  group.
	Then the notation $\operatorname{Hom}(G_{\mathbb{Q}},G)$ means the set of continuous homomorphism from the absolute Galois group $G_{\mathbb{Q}}$ of $\mathbb{Q}$ to a finite group $G$.
	Clearly the surjective ones up to conjugation correspond to the Galois $G$-fields up to isomorphism classes.
	
	\section{Local specifications}\label{sec: specifications}
	In this section, we establish the fundamental notions for our proof of the main results.
	Fix a base field $F$.
	Let us prove some properties of the Galois algebras.
	Note that some author may allow $0$ to be an idempotent.
	But we are in the context of the decomposition of primitive central idempotents.
	So, by an idempotent $e$ of a ring $R$, we mean that $0\neq e$ and $e^2=e$.
	\begin{lemma}\label{lemma: central idempotents}
		If $f:A\to B$ is an isomorphism of $F$-algebras, and $e$ is a primitive central idempotent of $A$, then $f(e)$ is a primitive central idempotent of $B$.
	\end{lemma}
	\begin{proof}
		Let $g:A\to B$ be any morphism of $F$-algebras.
		If $e$ is an idempotent of $A$, then $g(e)^2=g(e^2)=g(e)$.
		In other words, if the image of an idempotent is nonzero, then it is an idempotent.
		In addition, if $g$ is surjective, and $e$ is central in $A$, then $g(e)$ has to be central in $B$.
		
		Now assume that $f(e)=e'_1+e'_2$ admits a decomposition into a sum of two central idempotents in $B$.
		Note that $f^{-1}:B\to A$ is a surjective morphism of $F$-algebras.
		We have $e=f^{-1}(f(e))=f^{-1}(e_1')+f^{-1}(e_2')$.
		So, the preimage of a non-primitive idempotent is also non-primitive.
		And we are done for the proof.
	\end{proof}
	Note that when $f$ is not an isomorphism, the statement is clearly false, even if $A$ and $B$ are commutative $F$-algebras.
	Because $e$ may be contained in the kernel of $f$.
	\begin{lemma}\label{lemma: transitivity on idempotents}
		Let $(A/F,\varphi_A)$ be a Galois $(G,F)$-algebra.
		The induced action of $G$ on the set of all primitive idempotents of $A$ is transitive.
		That is, for each pair $(e_1,e_2)$ of primitive idempotents, there exists some $g\in G$ such that
		\begin{equation*}
			e_2=\varphi_A(g)(e_1).
		\end{equation*}
	\end{lemma}
	\begin{proof}
		Let $e_1$ be a primitive idempotent of $A$, and $\{e_1,\dots,e_l\}$ be the set of primitive idempotents generated by the action of $G$.
		That is, $\{e_1,\dots,e_l\}=\{\varphi_A(g)(e_1)\mid g\in G\}$.
		Note that by the Lemma~\ref{lemma: central idempotents}, each $\varphi_A(g)(e_1)$ must be a primitive central idempotent of $A$.
		Now define an element $e$ by
		\begin{equation*}
			e:=\sum_{i=1}^l e_i.
		\end{equation*}
		It is central because it is the sum of central elements.
		It is still an idempotent because $e_i^2=e_i$ and $e_ie_j=0$ for all $1\leq i\neq j\leq l$.
		Moreover, it is fixed by the action of $G$, that is, $\varphi_A(g)(e)=e$ for all $g\in G$.
		So, $e\in A^G=F$.
		But $F$ is a field, the only (central) idempotent is the identity $1_F=1_A$.
		Recall that $A\cong M_1\times\cdots\times M_n$, as an {\'e}tale $F$-algebra.
		This implies that $A$ is a semisimple $F$-algebra (Artinian with trivial Jacobson radical), and $1=e_1+\cdots+e_l$ is the decomposition of $1$ into primitive idempotents in $A$.
		So, the action of $G$ on $e_1$ generates all the primitive idempotents.
		And we are done.
	\end{proof}
	Let us first prove Theorem~\ref{thm: structure of Galois algebras} in Section~\ref{section: intro}.
	\begin{proof}[Proof of Theorem~\ref{thm: structure of Galois algebras}]
		Let $A$ be a Galois $(G,F)$-algebra.
		We first explain the structure of $A$ as an $F$-algebra.
		It is a finite dimensional {\'e}tale algebra by definition, so it must be of the form
		\begin{equation*}
			A\cong\prod_{i=1}^l M_i
		\end{equation*}
		where $M_i/F$ is a (finite) separable extension, for each $1\leq i\leq l$.
		By the Lemma~\ref{lemma: transitivity on idempotents}, we know that the action of $G$ is transitive on the idempotents, hence $M_i$-s are all isomorphic as fields.
		Let $n:=\lvert G\rvert=[A:F]$, and $m:=[M_i:F]$, then we have
		\begin{equation*}
			n=ml.
		\end{equation*}
		Fix a primitive idempotent $e$ of $A$ and let $G_e:=\operatorname{Stab}_G(e)$.
		We have $l=\lvert G/G_e\rvert$ and $m=\lvert G_e\rvert$.		
		Clearly for each $g\in G_e$, we have $g(eA)=eA$, so there is a group homomorphism $\varphi_{eA}:G_e\to\operatorname{Aut}_F(eA)$ defined by $\varphi_{eA}(g):=\varphi_A(g)\vert_{eA}$.
		
		Consider the morphism of $F$-algebras $\pi_e:A\to eA$ defined by $x\mapsto e\cdot x$.
		Clearly $\pi_e$ is surjective.
		Let $ex\in(eA)^{G_e}$.
		Define
		\begin{equation*}
			x':=\frac{1}{\lvert G_e\rvert}\sum_{g\in G}\varphi_A(g)(ex)\in A.
		\end{equation*}
		By the Lemma~\ref{lemma: central idempotents} above, the element $\varphi_A(g)(e)$ is some primitive idempotent.
		And $\varphi_A(g)(e)=e$ if and only if $g\in G_e$.
		Therefore $ex'=ex$, that is, $x'$ is a preimage of $ex$ under $\pi_e$.
		By the construction of $x'$, it is fixed by the action of $G$.
		So $x'\in A^G=F$, which implies that $ex\in\pi(F)$.
		Conversely, it is clear that $\pi(F)\subseteq(eA)^{G_e}$.
		This shows that $(eA)^{G_e}=\pi(F)\cong F$, which already implies that $eA/F$ is Galois, for the subgroup $\varphi_{eA}(G_e)$ of $\operatorname{Aut}_F(eA)$ fixes $F$.
		And this also implies that $\varphi_{eA}:G_e\to\operatorname{Gal}(eA/F)$ is surjective.		
		Recall that $m=\lvert G_e\rvert$, so the map $\varphi_{eA}$ has to be injective.
		Therefore, the algebra $(eA/F,\varphi_{eA})$ is a Galois $G_e$-field extension.
		
		We define an $F$-morphism $\tilde{f}:\operatorname{Ind}^G_{G_e}eA\to A$ by the universal property of the functor $\operatorname{Ind}^G_{G_e}$.
		To be precise, the identity map $1_{eA}:eA\to eA$ is an isomorphism of Galois $G_{e}$-extensions over $F$, hence also $G_{e}$-equivariant.
		Let us define $f:eA\to A$ by composing $1_{eA}$ and $eA\hookrightarrow A$.
		Define $i:eA\to\operatorname{Ind}^G_{G_e}eA$ by $ex\mapsto\tau_{ex}$, where $\tau_{ex}\in\operatorname{Map}(G,A)$ is the map defined by
		\begin{equation*}
			\tau_{ex}(g)=\left\{\begin{aligned}
				& gex & \text{ if }g\in G_{e} \\
				& 0 & \text{ else if }g\notin G_{e}.
			\end{aligned}\right.
		\end{equation*}
		It is clear that $\tau_{ex}\in\operatorname{Ind}^{G}_{G_{e}}eA$, and $i$ is a morphism of Galois $(H,F)$-algebras.
		By the universal property of $\operatorname{Ind}^G_{G_e}$, there is a unique $G$-morphism $\tilde{f}:\operatorname{Ind}^G_{G_e}eA\to A$ such that $\tilde{f}\circ i=f$.
		In particular, for each simple tensor $g\otimes m$, we have $\tilde{f}(g\otimes m)=\varphi_A(g)\cdot m$.
		The map $\tilde{f}$ is automatically $G$-equivariant by its definition.
		And it is surjective, for $G\cdot eA=A$ (transitivity on idempotents).
		The injectivity follows from the comparison of dimensions:
		The algebras $A$ and $\operatorname{Ind}^G_{G_e}eA$ both have dimension $\lvert G\rvert$ over $F$, so $\tilde{f}$ has to be injective if it is surjective.
		So $\tilde{f}$ gives an isomorphism between $\operatorname{Ind}^G_{G_e}eA$ and $A$ as Galois $(G,F)$-algebras.
		And we are done for the proof.
	\end{proof}	
	In most cases, the Galois group is fixed, just like our set-up at the beginning of this section: we fix a finite group $G$.
	However in order to describe the isomorphism of Galois algebras, we are faced with the change of groups.
	Namely, all the components $eA$ of $A$ are isomorphic as fields, and they have isomorphic Galois groups.
	But their Galois groups are represented by subgroups conjugate in $G$.
	Also, it is not enough for us to know that they are isomorphic as fields.
	We need to show that they are ``conjugate by $G$ in $A$'' in a strict sense.
	So, let us generalize the Definition~\ref{def: Galois algebras} a little bit as follows.
	\begin{definition}\label{def: category of Galois algebras}
		Define the category of Galois $F$-algebras as follows.
		The set of objects consists of pairs $(A,G)$ where $G$ is a finite group and $A$ is a Galois $(G,F)$-algebra.
		A morphism between two Galois $F$-algebras $(A,G)$ and $(B,H)$ is a pair $(f,\tau)$ where $f:A\to B$ is a morphism of $F$-algebras and $\tau:G\to H$ is a group homomorphism such that the following diagram is commutative
		\begin{equation*}
			\begin{tikzcd}
				A\arrow[d,"f"]\arrow[r,"\varphi_A(g)"]&A\arrow[d,"f"]\\
				B\arrow[r,"\varphi_B(\tau(g))"]&B
			\end{tikzcd}
		\end{equation*}
		for all $g\in G$.
		That is, for all $g\in G$ and $a\in A$, we have
		\begin{equation*}
			f(g\cdot a)=\tau(g)\cdot f(a).
		\end{equation*}
	\end{definition}
	Clearly, when restricted to Galois extensions over $F$, say $K_1/F$ and $K_2/F$, if they are isomorphic as field extensions over $F$, then there exists isomorphisms of the number equal to $[K_1:F]$ in the category of Galois $F$-algebras.
	Otherwise, there is no morphism between them.
	This shows that this category is a generalization of the notion of category of finite dimensional Galois extensions over $F$.
	
	For the other direction, the category of Galois $(G,F)$-algebras is the subcategory where the group $G$ is fixed and the morphism of groups is taken to be the identity $1_G$ of $G$.
	For example, isomorphic Galois extensions and isomorphic Galois $G$-extensions are different in general.
	Using this notation, we could describe the isomorphisms between the components of a Galois $(G,F)$-algebra in a more detailed way.
	\begin{lemma}\label{lemma: isomorphisms of G-algebras}
		Let $(A,\varphi_A)$be a Galois $(G,F)$-algebras.
		For each pair of primitive idempotents $(e,e')$, and for each $g\in G$ such that $e'=\varphi_A(g)(e)$, we have an isomorphism of Galois $F$-algebras $(\varphi_A(g),c_g):(eA,G_e)\to(e'A,G_{e'})$, where $G_e$ and $G_{e'}$ are the corresponding stabilizers, and $\varphi_A(g):eA\to e'A,ex\mapsto\varphi_A(g)(ex)$ is the field isomorphism induced by $g$, and $c_g:G_e\to G_{e'},h\mapsto ghg^{-1}$ is the conjugation by $g$.
	\end{lemma}
	\begin{proof}
		First of all, by the Lemma~\ref{lemma: transitivity on idempotents}, we know that there exists some $g\in G$ such that $\varphi_A(g)(e)=e'$.
		If $G_e$ is the stabilizer of $e$ in $G$, then clearly the conjugate $G_{e'}=gG_eg^{-1}$ is the stabilizer of $e'$.
		So the maps in the statement are well-defined.
		
		It suffices to verify that $(\varphi_A(g),c_g)$ is really an isomorphism of the Galois $F$-algebras.
		Since $\varphi_A(g)$ is an element of $\operatorname{Aut}_F(A)$, it is clear that $\varphi_A(g)$ restricted to $eA$ induces an isomorphism of field extensions $eA\cong e'A$ over $F$.
		It is also clear that $c_g:G_e\to G_{e'}$ is a group isomorphism.
		Now let $\sigma$ be an element of $G_e$, and $x\in eA$.
		We have
		\begin{equation*}
			\begin{aligned}
				\varphi_A(g)\bigl(\varphi_{eA}(\sigma)(x)\bigr)=&\varphi_A(g\sigma)(x)
				\\
				=&\varphi_A(g\sigma g^{-1}g)(x)
				\\
				=&\varphi_{e'A}\bigl(c_g(\sigma)\bigr)\bigl(\varphi_A(g)(x)\bigr).
			\end{aligned}
		\end{equation*}
		In short, we've shown that
		\begin{equation*}
			\varphi_A(g)(\sigma\cdot x)=c_g(\sigma)\cdot\varphi_A(x).
		\end{equation*}
		This shows that $\varphi_A(g)$ and $c_g$ are compatible.
		So the pair $(\varphi_A(g),c_g)$ is indeed an isomorphism of the Galois $F$-algebras, and we are done.
	\end{proof}
	Now we could prove a criterion for two Galois $(G,F)$-algebras being isomorphic.
	\begin{proposition}\label{prop: isomorphisms of G-algebras}
		Let $(A,\varphi_A)$ and $(B,\varphi_B)$ be two Galois $(G,F)$-algebras.
		They are isomorphic as Galois $(G,F)$-algebras if and only if there exists primitive idempotents $e\in A$ and $e'\in B$ such that the Galois $F$-algebras $(eA,G_e)$ and $(e'B,G_{e'})$ are isomorphic by a conjugation from $G$.
		That is, there exists some $g\in G$ such that the pair $(f:eA\to e'B,c_g:G_e\to G_{e'})$ is an isomorphism, where $c_g(h)=ghg^{-1}$ is the conjugation by $g$.
	\end{proposition}
	\begin{proof}
		First assume that $(A,\varphi_A)\cong(B,\varphi_B)$ as Galois $(G,F)$-algebras.
		By definition, there exists a morphism $f:A\to B$ of $F$-algebras, such that $f$ is bijective and $G$-equivariant.
		Let $e$ be any primitive idempotent of $A$.
		By the Lemma~\ref{lemma: central idempotents}, we know that $e':=f(e)$ is a primitive idempotent of $B$.
		Since $f$ is $G$-equivariant, we see that $H:=G_e=G_{e'}$.
		Then we see that $(f\vert_{eA}:eA\to e'B,1_{H}:H\to H)$ is automatically an isomorphism of the Galois $F$-algebras.
		
		For the opposite direction, let $(eA,G_e)$ and $(e'B,G_{e'})$ and $(f:eA\to e'B,c_g:G_e\to G_{e'})$ be as in the statement.
		By the Lemma~\ref{lemma: isomorphisms of G-algebras}, we see that
		\begin{equation*}
			(\varphi_B(g^{-1}),c_{g^{-1}}):(e'B,G_{e'})\to((g^{-1}e')B,G_{g^{-1}e'})
		\end{equation*} 
		is an isomorphism of Galois $F$-algebras.
		Note also that $H:=G_{g^{-1}e}=G_e$.
		This implies that $(\varphi_B(g^{-1})\circ f,1_{H}):(eA,H)\to(g^{-1}e'B,H)$ is an isomorphism of the Galois $F$-algebras.
		Since the morphism of groups is just the identity map of $H$, it reduces to an isomorphism of Galois $(H,F)$-field extensions: 
		\begin{equation*}
			\varphi_B(g^{-1})\circ f:(eA,\varphi_{eA})\to(g^{-1}e'B,\varphi_{g^{-1}e'B}).
		\end{equation*}
		In particular, the map $f':=\varphi_B(g^{-1})\circ f$ is $H$-equivariant.
		It induces an isomorphism
		\begin{equation*}
			\operatorname{Ind}^G_H f':\operatorname{Ind}^G_H eA\to\operatorname{Ind}^G_H g^{-1}e'B
		\end{equation*}
		of Galois $(G,F)$-algebras.
		Because $\operatorname{Ind}^G_H$ is a functor, and the resulting map $\operatorname{Ind}^G_H f'$ is a morphism of $F$-algebras that is $G$-equivariant with an inverse $\operatorname{Ind}^G_H (f')^{-1}$.
		In the Theorem~\ref{thm: structure of Galois algebras} we have shown that $\operatorname{Ind}^G_H eA\cong A$ as Galois $(G,F)$-algebras, and similarly $\operatorname{Ind}^G_H g^{-1}e'B\cong B$.
		Therefore, $A$ and $B$ are isomorphic as Galois $(G,F)$-algebras.
	\end{proof}
	\begin{remark}
		The proof shows that if $(A,\varphi_A)\cong(B,\varphi_B)$ are isomorphic as Galois $(G,F)$-algebras, then there must be idempotents $e$ and $e'$ of $A$ and $B$ respectively such that $G_e=G_{e'}$ and $(eA,\varphi_A)\cong(e'B,\varphi_B)$ as Galois $G_e$-extensions over $F$.
	\end{remark}
	Roughly speaking, the Galois $(G,F)$-algebras are determined by the conjugacy classes (in $G$) of its components.
	An immediate application of the proposition is when the algebra is a Galois $G$-field extension.
	We obtain the following description to classify different actions of $G$ on the field extensions.
	\begin{corollary}\label{cor: Galois G-extensions}
		Let $K/F$ be a fixed Galois field extension such that $\operatorname{Gal}(K/F)\cong G$.
		Given two group isomorphisms $\varphi_1,\varphi_2:G\to\operatorname{Gal}(K/F)$, the two Galois algebras $(K,\varphi_1)$ and $(K,\varphi_2)$ are isomorphic if and only if there exists some $g\in G$ such that
		\begin{equation*}
			\varphi_2=\varphi_1\circ c_g,
		\end{equation*}
		where $c_g(h)=ghg^{-1}$ is the conjugation by $g$.
	\end{corollary}
	\begin{proof}
		Given $g\in G$ such that $\varphi_2=\varphi_1\circ c_g$.
		Let $\sigma:=\varphi_1^{-1}(g)$ and one can check that $\sigma:(K,\varphi_1)\to(K,\varphi_2)$ is an isomorphism of Galois algebras.
		Conversely, if $\sigma\in\operatorname{Gal}(K/F)$ is an isomorphism of Galois algebras, then let $g:=\varphi_1(\sigma)$ and one can check that $\varphi_2=\varphi_1\circ c_g$.
	\end{proof}
	When $G$ is finite abelian, this result also explains the difference between the set of isomorphism classes of abelian $G$-fields and the set of isomorphism classes of abelian fields whose Galois group is isomorphic to $G$.
	The latter corresponds to surjective continuous homomorphisms $\operatorname{C}_{\mathbb{Q}}\to G$ where $\operatorname{C}_{\mathbb{Q}}$ is the id{\`e}les class group of $\mathbb{Q}$.
	While the former corresponds to the open normal subgroups $N$ of $\operatorname{C}_{\mathbb{Q}}$ such that $\operatorname{C}_{\mathbb{Q}}/N\cong G$.
	Generally speaking, the continuous homomorphisms are easier to manipulate.
	Also, the following statement explains why we prefer the notion of Galois algebra from another aspect.
	That is, we could establish a one-to-one correspondence between the isomorphism classes of the Galois $(G,F)$-algebras and the conjugacy classes of continuous homomorphisms $G_{F}\to G$.
	\begin{proposition}\label{prop: 1-1 correspondence for G-structured algebras}
		Fix a finite group $G$ and a field $F$.
		There is a one-to-one correspondence between the following two sets
		\begin{equation*}
			\{\text{Galois }(G,F)\text{-algebras up to isomorphism}\}\leftrightarrow\operatorname{Hom}(G_F,G)/\sim
		\end{equation*}
		where $G_F$ is the absolute Galois group of $F$, and $\sim$ means the equivalence relation induced by the conjugate of $G$, that is, $\rho,\chi:G_F\to G$ are equivalent if there is some $g\in G$ such that $\chi=g\rho g^{-1}$.
	\end{proposition}
	\begin{proof}
		Let $\bar{F}$ be the algebraic closure of $F$, and $\rho:G_F\to G$ be a continuous homomorphism with image $H$.
		Then $K:=\bar{F}^{\ker\rho}$ is a Galois $H$-extension over $F$, where the map $\varphi_K:H\to\operatorname{Aut}_F(K)$ is induced by $\rho$.
		And $A:=\operatorname{Ind}^G_H K$ is naturally a Galois $(G,F)$-algebra.
		
		If $\chi:G_F\to G$ is another continuous homomorphism such that there exists some $g\in G$ for all $\sigma\in G_F$ we have $\chi(\sigma)=g\rho(\sigma)g^{-1}$, then $K=\bar{F}^{\ker\chi}$ is a Galois $gHg^{-1}$-extension over $F$.
		This implies that $(K,H)$ and $(K,gHg^{-1})$ are isomorphic as Galois $F$-algebras by $(1_K,c_g)$.
		Let $B:=\operatorname{Ind}^G_{gHg^{-1}}K$.
		By the Proposition~\ref{prop: isomorphisms of G-algebras}, we see that $A$ and $B$ are isomorphic as Galois $(G,F)$-algebras.
		
		Therefore, we have a well-defined map
		\begin{equation*}
			\Phi:\operatorname{Hom}(G_F,G)/\sim\to\{\text{Galois }(G,F)\text{-algebras}\}/\cong,\quad\varphi\mapsto\operatorname{Ind}^G_{\varphi(G_F)}\bar{F}^{\ker\varphi}.
		\end{equation*} 
		Let $A$ be a Galois $(G,F)$-algebra.
		By the Theorem~\ref{thm: structure of Galois algebras}, if $e$ is a primitive idempotent, then $eA$ is a Galois $(G_e,F)$-field extension.
		As an algebraic extension over $F$, we could choose an embedding $eA\hookrightarrow\bar{F}$.
		Then we have a surjective continuous homomorphism $\rho_{eA}:G_F\to\operatorname{Aut}_F(eA)$,
		and the isomorphism $\varphi:G_e\to\operatorname{Aut}_F(eA)$ induces a continuous homomorphism $\rho_A:=\varphi_A^{-1}\circ\rho_{eA}:G_F\to G$ with image $G_e$.
		By the Corollary~\ref{cor: Galois G-extensions}, we see that different embeddings $eA\hookrightarrow\bar{F}$ induces conjugate surjective homomorphisms $G_F\to\operatorname{Aut}_F(eA)$.
		Therefore the map $A\mapsto\rho_A$ is defined up to conjugation.
		That is, the algebra $A$ corresponds to a class of homomorphisms.
		
		If $B$ is another Galois $(G,F)$-algebra that is isomorphic to $A$, then by the Proposition~\ref{prop: isomorphisms of G-algebras}, we may assume without loss of generality that $e$, resp. $e'$, is a primitive idempotent of $A$, resp. of $B$, such that $G_e=G_{e'}$ and $f:(eA,\varphi_{A})\to(e'B,\varphi_{B})$ is an isomorphism of Galois $(G_e,F)$-algebras.
		Given any embedding $\sigma:e'B\hookrightarrow\bar{F}$, we see that $\sigma\circ f:eA\hookrightarrow\bar{F}$ is an embedding such that the induced maps $\rho_A=\varphi_A^{-1}\circ\rho_{eA}$ and $\rho_B=\varphi_B^{-1}\circ\rho_{eB}$ coincide.
		This implies that $A$ and $B$ correspond to the same class of continuous homomorphisms.
		So, we have a well-defined map
		\begin{equation*}
			\Psi:\{\text{Galois }(G,F)\text{-algebras}\}/\cong\to\operatorname{Hom}(G_F,G)/\sim,\quad(A,\varphi_A)\mapsto\rho_A.
		\end{equation*}
		Then it suffices to check that $\Phi$ and $\Psi$ are inverse to each other.
		Let $\rho:G_F\to G$ be a continuous homomorphism with image $H$ and kernel $G_K$.
		Up to the equivalence relations, we have $\Phi(\rho)=\operatorname{Ind}^G_H K$, denoted by $A:=\Phi(\rho)$.
		The Galois $(G,F)$-algebra $A$ has a component $K$ which is a Galois $H$-extension over $F$.
		It is already a subfield of $\bar{F}$ by its construction, so $\varphi_K:H\to\operatorname{Aut}_F(K)$ induces the map $\rho_A:G_F\to G$ by composing with the restriction of Galois actions $G_F\to\operatorname{Aut}_F(K)$.
		But this is exactly the original map $\rho$, for the map $\varphi_K:H\to\operatorname{Aut}_F(K)$ is induced by $\rho$.
		That is, the action of $H$ on the field extension $K/F$ is given by $\rho:G_F\to G$.
		So, we see that $\Psi\circ\Phi$ is the identity map.
		
		Now let $(A,\varphi_A)$ be a Galois $(G,F)$-algebra, $e$ be a primitive idempotent, and $\sigma:eA\hookrightarrow\bar{F}$ be a choice of embedding with the induced surjective map $\rho_{eA}:G_F\to\operatorname{Aut}_F(eA)$.
		By definition, $\rho_A=\Psi(A)=\varphi_A^{-1}\circ\rho_{eA}$.
		Clearly, $K:=\bar{F}^{\ker\rho_A}=\sigma(eA)$, and $K$ is a Galois $(G_e,F)$-field by the map $\rho_A:G_F\to G$.
		So, we see that $\sigma:eA\to K$ is an isomorphism of Galois $G_e$-field extensions over $F$, and $\Phi(\rho_A)=\operatorname{Ind}^G_{G_e}K$ is isomorphic to $A$ as Galois $(G,F)$-algebras by the Proposition~\ref{prop: isomorphisms of G-algebras}.
		This implies that $\Phi\circ\Psi$ is also the identity (of the opposite direction).
		And we are done.
	\end{proof}
	When $G$ is abelian, conjugation is trivial, and we have the following direct corollary.
	\begin{corollary}
		If $G$ is a finite abelian group, then we have the one-to-one correspondence between the following two sets
		\begin{equation*}
			\{\text{isomorphism classes of Galois }(G,F)\text{-algebra}\}\leftrightarrow\operatorname{Hom}(G_F,G).
		\end{equation*}
		In particular, the surjective ones correspond to the Galois $G$-extensions over $F$.
	\end{corollary}
	The general correspondence between the Galois $(G,F)$-algebras and homomorphisms $G_F\to G$ implies that the counting function and the generalized discriminant (see the Definition~\ref{def: generalized discriminant}) could be defined for continuous homomorphisms.
	\begin{definition}\label{def: generalized discriminant of morphisms}
		Fix a finite group $G$ and a number field $K$.
		\begin{enumerate}
			\item For each prime $\mathfrak{p}\in\mathcal{P}_K$, 
			define $G_{K_{\mathfrak{p}}}$ to be the absolute Galois group of $K_{\mathfrak{p}}$, 
			$G_{K_{\mathfrak{p}}}^t\subseteq G_{K_{\mathfrak{p}}}$ be the Galois group of the maximal tamely ramified extension over $K_{\mathfrak{p}}$, 
			and $I_{K_{\mathfrak{p}}}^t\subseteq G_{K_{\mathfrak{p}}}^{t}$ be the tame inertia subgroup.
			Let $y_{\mathfrak{p}}^{t}$ be the topological generator of $I_{K_{\mathfrak{p}}}^{t}$.
			\item Let $c_G:G\to\mathbb{R}_{\geq0}$ be a weight.
			For each prime $\mathfrak{p}\nmid\lvert G\rvert\infty$ of $K$, and for each continuous map $\rho_{\mathfrak{p}}:G_{K_{\mathfrak{p}}}\to G$, define the local generalized discriminant by
			\begin{equation*}
				C_{\mathfrak{p}}(\rho_{\mathfrak{p}}):=\mathfrak{Np}^{c_G(\rho_{\mathfrak{p}}(y_{\mathfrak{p}}^{t}))}.
			\end{equation*}
			\item For each prime $\mathfrak{p}$ of $K$, a local generalized discriminant is a map $C_{\mathfrak{p}}:\operatorname{Hom}(G_{K_{\mathfrak{p}}},G)\to\mathbb{R}_{+}$ such that $C_{\mathfrak{p}}(\rho_{\mathfrak{p}})=C_{\mathfrak{p}}(\chi_{\mathfrak{p}})$ whenever $\rho_{\mathfrak{p}}$ and $\chi_{\mathfrak{p}}$ are conjugate to each other by $G$.
			Let $c_{G}:G\to\mathbb{R}_{\geq 0}$ be a weight.
			When $\mathfrak{p}\nmid\lvert G\rvert\infty$, then $\rho_{\mathfrak{p}}:G_{K_{\mathfrak{p}}}\to G$ factors through $G_{K_{\mathfrak{p}}}^t$.
			If in this case, we have
			\begin{equation*}
				C_{\mathfrak{p}}(\rho_{\mathfrak{p}})=\mathfrak{Np}^{c_G(\rho_{\mathfrak{p}}(y_{\mathfrak{p}}^{t}))},
			\end{equation*}
			then we say that $C_{\mathfrak{p}}$ is a local generalized discriminant with respect to the weight $c_{G}$.
		\end{enumerate}		
	\end{definition}
	For primes $\mathfrak{p}\nmid\lvert G\rvert\infty$, we could define the local generalized discriminant not only for continuous homomorphisms $\rho_{\mathfrak{p}}:G_{K_{\mathfrak{p}}}^t\to G$, but also for closed normal subgroups that contain $y_{\mathfrak{p}}^t$.
	It is clear that the Definition~\ref{def: generalized discriminant of morphisms} is a generalization of the Definition~\ref{def: generalized discriminant} in the following sense.
	\begin{proposition}\label{prop: generalized discriminant of morphisms}
		Fix a finite group $G$ and a number field $K$.
		Let $c_{G}:G\to\mathbb{R}_{\geq0}$ be a weight of $G$, and $C:\mathcal{C}(G,K)\to\mathbb{R}_{+}$ be a generalized discriminant with respect to $c_{G}$.
		If $L/K$ is a Galois $G$-extension that corresponds to the continuous surjective map $\rho:G_K\to G$, then for each prime $\mathfrak{p}$ of $K$, there exists a local generalized discriminant $C'_{\mathfrak{p}}:\operatorname{Hom}(G_{K_{\mathfrak{p}}},G)\to\mathbb{R}_{+}$ such that
		\begin{equation*}
			C_{\mathfrak{p}}(L_{\mathfrak{p}})=C'_{\mathfrak{p}}(\rho_{\mathfrak{p}})
		\end{equation*}
		where $\rho_{\mathfrak{p}}$ is the continuous map (up to conjugate) corresponding to $L_{\mathfrak{p}}$.
		In particular, if $\mathfrak{p}\nmid\lvert G\rvert\infty$, then
		\begin{equation*}
			C_{\mathfrak{p}}(L_{\mathfrak{p}})=\mathfrak{Np}^{c_{G}(\rho_{\mathfrak{p}}(y^{t}_{\mathfrak{p}}))}.
		\end{equation*}
	\end{proposition}
	\begin{proof}
		The case when $\mathfrak{p}\nmid\lvert G\rvert\infty$ follows from the fact that if $L_{\mathfrak{p}}$ and $\rho_{\mathfrak{p}}$ correspond to each other, then the inertia generator $y_{\mathfrak{p}}$ of $L_{\mathfrak{p}}$ is exactly the image of $y_{\mathfrak{p}}^t$ (up to conjugate).
		
		For each $\mathfrak{p}\mid\lvert G\rvert\infty$ of $K$, we could define 
		\begin{equation*}
			C'_{\mathfrak{p}}(\rho_{\mathfrak{p}}):=C_{\mathfrak{p}}(L_{\mathfrak{p}})
		\end{equation*}
		whenever $\rho_{\mathfrak{p}}$ is the continuous homomorphism corresponding to $L_{\mathfrak{p}}$.		
		If a Galois $(G,K_{\mathfrak{p}})$-algebra $\Sigma_{\mathfrak{p}}$ (up to isomorphism) never shows up as a local specification of the global field extension $L/K$, then we could simply assign the value $1$ to it.
		That is, if the isomorphism class $\Sigma_{\mathfrak{p}}$ of the Galois $(G,K_{\mathfrak{p}})$-algebras corresponds to $\rho_{\mathfrak{p}}$, and if for each Galois $G$-extension $L/K$ the specification $L_{\mathfrak{p}}$ at $\mathfrak{p}$ is not isomorphic to $\Sigma_{\mathfrak{p}}$, then define
		\begin{equation*}
			C'_{\mathfrak{p}}(\rho_{\mathfrak{p}})=1.
		\end{equation*}
		So, by combining these two cases, for each $\mathfrak{p}$ of $K$, we have obtained a local generalized discriminant $C'_{\mathfrak{p}}$ for the equivalence classes of the continuous homomorphisms $\operatorname{Hom}(G_{K_{\mathfrak{p}}},G)$ such that 
		\begin{equation*}
			C_{\mathfrak{p}}(L_{\mathfrak{p}})=C'_{\mathfrak{p}}(\rho_{\mathfrak{p}}),
		\end{equation*}
		whenever $L_{\mathfrak{p}}$ and $\rho_{\mathfrak{p}}$ corresponds to each other.		
	\end{proof}
	This also explains the condition of the equivalence relation under conjugation and invertible powering in the weight function $c_{G}:G\to\mathbb{R}_{\geq 0}$.
	That is, if $g_{1}\sim g_{2}$, then $c_{G}(g_{1})=c_{G}(g_{2})$.
	Because the equivalent continuous maps $G_{K_{\mathfrak{p}}}\to G$ are conjugate to each other in $G$, and when we forget the Galois structure (e.g., continuous maps with the same kernel may not be conjugate to each other in $G$), the images of the topological generator $y_{\mathfrak{p}}^t$ must satisfy the the relation $y_{1}=y_{2}^{a}$ and $y_{2}=y_{1}^{b}$ for some integers $a$ and $b$.
	Since the discriminant cannot see the different Galois structures, the equivalence relation must cover both the conjugation and the invertible powering, hence the definition of the weight function.
	
	\section{Restricted ramification}\label{sec: restricted ramification}
	In this section, let $G$ be a transitive permutation group with $\operatorname{Stab}_G(1)$ being trivial.
	In other words, if $K/k$ is a $G$-extension, then it is a Galois $G$-extension.
	\begin{definition}\label{def: restricted ramification}
		Fix a number field $k/\mathbb{Q}$.
		Let $S$ be a finite set of primes of $k$ including the ones at infinity.
		Let $\Phi$ be a set of $G$-specifications $\Sigma=(\Sigma_{\mathfrak{p}})_{\mathfrak{p}\in S}$.
		We say that $K$ is a solution to $(G,k,\Phi)$ if there exists $\Sigma\in\Phi$ such that $K$ is a solution to $(G,k,\Sigma)$, 
		that is, there exists a $G$-specification $\Sigma\in\Phi$ at $S$ such that $K/k$ is a Galois $G$-extension unramified outside $S$ with $K_{\mathfrak{p}}\cong\Sigma_{\mathfrak{p}}$ for all $\mathfrak{p}\in S$.
		Denote the set of solutions to $(G,k,\Phi)$ by $\mathcal{C}(G,k,\Phi)$ or just $\mathcal{C}(\Phi)$ when we omit $G$ and $k$.
		When $\Phi=\{\Sigma\}$, we simply write $\mathcal{C}(\Phi)=\mathcal{C}(\Sigma)$.
		In addition, let $\mathcal{C}$ be a subset of $\mathcal{C}(G,k)$ with a generalized discriminant $C:\mathcal{C}\to\mathbb{R}_{+}$,
		For each positive real number $x$, let
		\begin{equation*}
			\mathcal{C}(x):=\{K/k\in\mathcal{C}\mid C(K)=x\}.
		\end{equation*}
		In other words, the condition that $C(K)=x$ is automatically a set of $G$-specifications when $C$ is a generalized discriminant.
	\end{definition}
	This is just the classical set-up of the restricted ramification problem in the context of the arithmetic statistics.
	See also Neukirch, Schmidt and Wingberg~\cite[Chapter X]{neukirch2013cohomology}	
	\begin{definition}
		Let $k$ be a number field.
		For each $p\in\mathcal{P}$, define
		\begin{equation*}
			S_p(k):=\{\mathfrak{p}\in\mathcal{P}_k:\mathfrak{p}\mid p\}.
		\end{equation*}
		For a finite subset $S\subseteq\mathcal{P}$ containing the infinite place, define
		\begin{equation*}
			S(k):=\bigcup_{p\in S}S_p(k).
		\end{equation*}		
	\end{definition}	
	Let us first give an upper bound estimate for the solutions to the abelian extensions.
	For this purpose, we recall some standard notations.
	For a number field $K$, for each prime $\mathfrak{p}$ of $K$, let $K_{\mathfrak{p}}$ be the completion of $K$ at $\mathfrak{p}$.
	Define
	\begin{equation*}
		\begin{aligned}
			U_{\mathfrak{p}}:=\left\{
			\begin{aligned}
				&\{x\in K_{\mathfrak{p}}:\lvert x\rvert_{\mathfrak{p}}=1\}&\quad &\text{if }\mathfrak{p}\text{ is finite};
				\\
				&\mathbb{R}_+^*&\quad&\text{if }\mathfrak{p}\text{ is infinite real};
				\\
				&\mathbb{C}^*&\quad &\text{if }\mathfrak{p}\text{ is infinite complex}.
			\end{aligned}\right.
		\end{aligned}
	\end{equation*}
	In other words, $U_{\mathfrak{p}}$ is the unit group of $K_{\mathfrak{p}}$.
	Let $J_K$ be the group of id{\`e}les of $K$, and $\operatorname{C}_K$ the id{\`e}les class group.
	Let $S$ be a finite set of primes containing the ones at infinity.
	Define
	\begin{equation*}
		J_K^{S}:=\prod_{\mathfrak{p}\in S}K_{\mathfrak{p}}^*\times\prod_{\mathfrak{p}\notin S}U_{\mathfrak{p}}
	\end{equation*} 
	to be the $S$-id{\`e}les,and define $K^S:=J_K^S\cap K^*$ to be the $S$-units of $K$.
	\begin{lemma}\label{lemma: restricted ramification for abelian extensions}
		Assume that $G$ is finite abelian.
		Fix a number field $k$.
		Let $\mathcal{C}:=\mathcal{C}(G,k)$.		
		For each finite subset $S\subseteq\mathcal{P}_k$ including the ones at infinity, and for each admissible $G$-specification $\Sigma:=(\Sigma_{\mathfrak{p}})_{\mathfrak{p}\in S}$ at $S$, we have
		\begin{equation*}
			\#\mathcal{C}(\Sigma)\leq\lvert\operatorname{Hom}(\operatorname{Cl}^{S}_{k},G)\rvert\leq\lvert\operatorname{Hom}(\operatorname{Cl}_{k},G)\rvert.
		\end{equation*}
	\end{lemma}
	\begin{proof}    	
		It suffices to show that $\#\mathcal{C}(\Sigma)\leq\lvert\operatorname{Hom}(\operatorname{Cl}^{S}_{k},G)\rvert$.
		Assume without loss of generality that $K/k$ is a solution to $\mathcal{C}^{\Sigma}$.
		The statement follows from the short exact sequence
		\begin{equation*}
			1\to J_k^{S}/k^{S}\to\operatorname{C}_k\to\operatorname{Cl}_k^{S}\to1.
		\end{equation*}
		Apply the functor $\operatorname{Hom}(\mathrm{-},G)$ and we have a long exact sequence
		\begin{equation*}
			1\to\operatorname{Hom}(\operatorname{Cl}_k^{S},G)\to\operatorname{Hom}(\operatorname{C}_k,G)\to\operatorname{Hom}(J_k^{S}/k^{S},G)\to\cdots.
		\end{equation*}
		Since $K/k$ is clearly a solution to $(G,\Sigma)$, let $\chi_K:\operatorname{C}_k\to G$ be the surjective continuous homomorphism that corresponds to $K$.
		Denote the restriction of $\chi_K$ to $J_k^S/k^S$ by $\chi_K^S$.
		
		Claim: For each solution $L\in\mathcal{C}(\Sigma)$, the restriction of $\chi_L$ to $J_k^S/k^S$ is the same as $\chi_K^S$.
		By Wood~\cite[Lemma 2.6]{wood2010probabilities}, there is a one-to-one correspondence between the set of the isomorphism classes of $G$-structured $k_{\mathfrak{p}}$-algebras and $\operatorname{Hom}(k_{\mathfrak{p}}^*,G)$.
		For each $\mathfrak{p}\in S$, the induced local map $\chi_{L,\mathfrak{p}}:k_{\mathfrak{p}}^*\to G$ corresponds to $K_{\mathfrak{p}}$, so it coincides with $\chi_{K,\mathfrak{p}}$.
		And for each $\mathfrak{p}\notin S$, the algebra $L_{\mathfrak{p}}/k_{\mathfrak{p}}$ is unramified, which implies that the induced map $\chi_{L,\mathfrak{p}}:U_{\mathfrak{p}}\to G$ must be trivial.
		Since we can represent the map $\chi_K^S$ by a continuous homomorphism $\prod_{\mathfrak{p}}\chi_{K,\mathfrak{p}}:J_k^S\to G$ that is trivial on $k^S$, that is, the map $\chi_K^S$ is determined by the local data.
		We see that $\chi_L^S$ must be the same as $\chi_K^S$.
		
		The claim shows that $\mathcal{C}(\Sigma)$ is a subset of the preimage of $\chi_{K}^S$.
		The size of the preimage is $\lvert\operatorname{Hom}(\operatorname{Cl}_k^S,G)\rvert$ by the exactness at $\operatorname{Hom}(\operatorname{C}_k,G)$.
		Therefore, we have the estimate
		\begin{equation*}
			\#\mathcal{C}(\Sigma)\leq\lvert\operatorname{Hom}(\operatorname{Cl}^{S}_{k},G)\rvert.
		\end{equation*}
	\end{proof}
	To generalize the estimate for the abelian extensions to the solvable extensions, we need some technical results.
	\begin{lemma}\label{lemma: local specification in field extensions}
		Fix a number field $k$, and a $G$-extension $K/k$.
		\begin{enumerate}
			\item Let $L/k$ be a $G$-extension such that $L_{\mathfrak{p}}\cong K_{\mathfrak{p}}$ as Galois $(G,k_{\mathfrak{p}})$-algebras for some prime $\mathfrak{p}\in\mathcal{P}_k$.
			For each a normal subgroup $N$ of $G$, we have
			\begin{equation*}
				(K^N)_{\mathfrak{p}}\cong (L^N)_{\mathfrak{p}}
			\end{equation*}
			as Galois $(G/N,k_{\mathfrak{p}})$-algebras.
			\item Let $H$ be a subgroup of $G$, and $u$ be a prime of $k$.
			For each prime $v$ of $K^H$ lying above $u$, we have
			\begin{equation*}
				K_{v}\cong
				\prod_{
				\substack{
					w\in\mathcal{P}_K\\
					w\mid v
				}}K_{w}
				\quad\text{and}\quad
				\prod_{
					\substack{
						v\in\mathcal{P}_{K^H}\\
						v\mid u
					}
				}K_{v}\cong K_{u},
			\end{equation*}
			where $K_{v}=K\otimes(K^H)_{v}$.
			\item If $K/k$ is unramified outside some finite subset $S\subseteq\mathcal{P}_k$ including the ones at infinity, then $\mathfrak{Nd}_{K/k}$ is determined by the local specification $(K_{\mathfrak{p}})_{\mathfrak{p}\in S}$.
		\end{enumerate}		
	\end{lemma}
	\begin{proof}
		(1):
		Let $\rho:G_{k}\to G$ be the surjective group homomorphism that corresponds to $K$.
		Clearly, the $G$-specification $K_{\mathfrak{p}}$ at $\mathfrak{p}$ corresponds to the continuous map $\rho_{\mathfrak{p}}:G_{k_{\mathfrak{p}}}\to G$.
		Since $N$ is a normal subgroup, we have induced maps
		\begin{equation*}
			\bar{\rho}:G_{k}\to G/N\quad\text{and}\quad\bar{\rho}_{\mathfrak{p}}:G_{k_{\mathfrak{p}}}\to G/N.
		\end{equation*}
		Clearly the map $\bar{\rho}$ corresponds to the Galois $G/N$-extension $K^N/k$.
		So the map $\bar{\rho}_{\mathfrak{p}}$ corresponds to the $G$-specification $(K^{N})_{\mathfrak{p}}$ by the Proposition~\ref{prop: 1-1 correspondence for G-structured algebras}.
		If $L$ is a Galois $G$-extension over $k$ with the corresponding continuous homomorphism $\chi:G_{k}\to G$ such that $L_{\mathfrak{p}}\cong K_{\mathfrak{p}}$, then by the Proposition~\ref{prop: 1-1 correspondence for G-structured algebras} we know that $\chi_{\mathfrak{p}}$ is conjugate to $\rho_{\mathfrak{p}}$ in the sense that there exists some $g\in G$ such that for all $\sigma\in G_{k_{\mathfrak{p}}}$ we have
		\begin{equation*}
			\chi_{\mathfrak{p}}(\sigma)=g\rho_{\mathfrak{p}}(\sigma)g^{-1}.
		\end{equation*}
		This implies immediately that for each $\sigma\in G_{k_{\mathfrak{p}}}$, we have
		\begin{equation*}
			\bar{\chi}_{\mathfrak{p}}(\sigma)=\bar{g}\bar{\rho}_{\mathfrak{p}}(\sigma)\bar{g}^{-1}
		\end{equation*}
		where $\bar{g}$ is the image of $g$ in $G/N$.
		By the Proposition~\ref{prop: 1-1 correspondence for G-structured algebras}, we know that $(L^{N})_{\mathfrak{p}}\cong(K^{N})_{\mathfrak{p}}$ as Galois $(G/N,k_{\mathfrak{p}})$-algebras.
		
		(2):
		This follows easily from the structure of $K_{v}$, which is isomorphic to
		\begin{equation*}
			\prod_{w\mid v}K_{w}.
		\end{equation*}
		And $K_{v}\cong\prod_{w\mid v}K_{w}$.
		
		(3):
		This follows from the fact that the discriminant could be computed locally, that is,
		\begin{equation*}
			\mathfrak{Nd}_{K/k}=\prod_{\mathfrak{p}\in\mathcal{P}_k}\prod_{\mathfrak{P}\mid\mathfrak{p}}\mathfrak{Nd}_{K_{\mathfrak{P}}/k_{\mathfrak{p}}}.
		\end{equation*}
		For each $\mathfrak{p}\in S$, the data of the local specification $K_{\mathfrak{p}}/k_{\mathfrak{p}}$ determines the set of all the places $\mathfrak{P}$ above $\mathfrak{p}$ with the corresponding the local extensions $K_{\mathfrak{P}}/k_{\mathfrak{p}}$.
		Therefore the local relative discriminant are fixed, hence the relative discriminant.
	\end{proof}
	\begin{lemma}\label{lemma: restricted ramification for sovlable extensions}
		\begin{enumerate}
			\item Let $A$ and $B$ be two finite abelian $p$-groups.
			We have
			\begin{equation*}
				\lvert \operatorname{Hom}(A,B)\rvert\leq\lvert A\rvert^{\operatorname{rk}_p B}.
			\end{equation*}
			\item Fix a finite abelian group $A$, and a number field $[K:\mathbb{Q}]=n$.
			There exist positive constants $c$ and $N$ depending only on $n$ and $A$ such that 
			\begin{equation}\label{eqn: restricted ramification for sovlable extensions lemma 1}
				\lvert\operatorname{Hom}(\operatorname{Cl}_K,A)\rvert\leq c d_K^{N},
			\end{equation}
			where $d_K$ is the absolute discriminant.
		\end{enumerate}
	\end{lemma}
	\begin{proof}
		(1):
		For finite cyclic $p$-groups, we see that
		\begin{equation*}
			\operatorname{Hom}(\mathbb{Z}/p^a\mathbb{Z},\mathbb{Z}/p^b\mathbb{Z})\cong\mathbb{Z}/p^{\min\{a,b\}}\mathbb{Z}.
		\end{equation*}
		In general, write $A\cong\prod_{i=1}^{\operatorname{rk}_p A}\mathbb{Z}/p^{a_i}\mathbb{Z}$ and $B\cong\prod_{i=1}^{\operatorname{rk}_p B}\mathbb{Z}/p^{b_i}\mathbb{Z}$, we have
		\begin{equation*}
			\operatorname{Hom}(A,B)\cong\prod_{i=1}^{\operatorname{rk}_p A}\prod_{j=1}^{\operatorname{rk}_p B}\mathbb{Z}/p^{\min\{a_i,b_j\}}\mathbb{Z}.
		\end{equation*}
		This implies that
		\begin{equation*}
			\lvert\operatorname{Hom}(A,B)\rvert\leq\prod_{j=1}^{\operatorname{rk}_p B}\prod_{i=1}^{\operatorname{rk}_p A}p^{a_i}=\lvert A\rvert^{\operatorname{rk}_p B}.
		\end{equation*}
		And we are done for (1).
		
		(2):
		By (1), we have
		\begin{equation*}
			\begin{aligned}
				\lvert\operatorname{Hom}(\operatorname{Cl}_K,A)\rvert=&\prod_{p}\lvert\operatorname{Hom}(\operatorname{Cl}_K[p^\infty],A[p^\infty])\rvert
				\\
				\leq&\prod_p\lvert\operatorname{Cl}_K[p^\infty]\rvert^{\operatorname{rk}_p A}
				\leq\lvert\operatorname{Cl}_K\rvert^{r_A},
			\end{aligned}
		\end{equation*}
		where $r_A=\sup_p\operatorname{rk}_pA$.
		By Minkowski bound of the size of the class group, for each $\varepsilon>0$, there exists a constant $c=c(n,\varepsilon)$ such that
		\begin{equation*}
			\lvert\operatorname{Cl}_K\rvert^{r_A}\leq c d_K^{\frac{1+\varepsilon}{2}r_A},
		\end{equation*}
		where $d_K$ is the absolute discriminant of $K/\mathbb{Q}$.
		And this shows the existence of $c$ and $N$.
	\end{proof}
	Let us prove an upper bound of solutions to the restricted ramification for solvable extensions.
	\begin{theorem}\label{thm: restricted ramification for solvable extensions}
		Assume that $G$ is solvable, and there exists $n\in\mathbb{Z}_+$ and a tower of subgroups of $G$:
		\begin{equation*}
			\{e_G\}=G_{l}\triangleleft G_{l-1}\triangleleft\cdots\triangleleft G_0=G
		\end{equation*}
		such that $G_{i+1}$ is normal in $G_i$ and $G_i/G_{i+1}$ is abelian for each $0\leq i\leq n-1$.
		Fix a number field $k$, and a finite set $S$ of rational primes including the infinity.
		For each $0\leq i\leq n-1$, let $c_i$ and $N_i$ be the constants from Lemma~\ref{lemma: restricted ramification for sovlable extensions} depending only on $n_i:=[k:\mathbb{Q}]\lvert G/G_i\rvert$ and $G_{i}/G_{i+1}$ such that for each number field $[L:\mathbb{Q}]=n_i$, we have
		\begin{equation*}
			\lvert\operatorname{Hom}(\operatorname{Cl}_L,G_i/G_{i+1})\rvert\leq c_id_L^{N_i}.
		\end{equation*}
		Let $K/k$ be a $G$-extension unramified outside $S$, and $\Sigma:=(K_{\mathfrak{p}})_{\mathfrak{p}\in S(k)}$.
		We have
		\begin{equation}\label{eqn: restricted ramification for solvable extensions thm 1}
			\#\mathcal{C}(\Sigma)\leq \prod_{i=0}^{n-1}c_id_{K_i}^{N_i},
		\end{equation}
		where $K_i:=K^{G_i}$.
	\end{theorem}
	\begin{proof}
		First of all, the expression~(\ref{eqn: restricted ramification for solvable extensions thm 1}) is well-defined in the sense that it is independent of the choice of $K\in\mathcal{C}(\Sigma)$.
		Because for each $K'\in\mathcal{C}(\Sigma)$, the extension $K'/k$ is unramified outside $S(k)$, and its local specification at $S(k)$ is given by $\Sigma$.
		So, we have $d_{K'}=d_{K}$.
		Moreover, Lemma~\ref{lemma: local specification in field extensions}(2) shows that
		\begin{equation*}
			\mathfrak{Nd}_{K/K_i}=\mathfrak{Nd}_{K'/K'_i},
		\end{equation*}
		because as $G_i$-extensions, they have exactly the same local specification at $S(K_i)$.
		Since $d_{K}=\mathfrak{Nd}_{K/K_i}\cdot d_{K_i}^{\lvert G_i\rvert}$, we see that $d_{K_i}=d_{K'_i}$.
		And the formula does not depend on the choice of $K\in\mathcal{C}(\Sigma)$.
		
		We prove the statement by induction on $n$, the length of the tower of subgroups of the solvable group $G$.
		If $G$ is abelian, then this reduces to the case of abelian extensions, and the upper bound simply comes from the Minkowski bound for the size of the class group.
		In other words, the statement is true when $n=1$.
		
		Now assume that the statement is true for $1,\dots,n-1$, where $n\geq2$.
		The subgroup $G_1$ is normal in $G$, and it is also solvable with the tower
		\begin{equation*}
			\{e_G\}=G_{l}\triangleleft G_{l-1}\triangleleft\cdots\triangleleft G_1.
		\end{equation*}
		By Lemma~\ref{lemma: local specification in field extensions}, for each $K\in\mathcal{C}(\Sigma)$, the field extension $K_1/k$ has the same local specification at $S(k)$, and $K/K_1$ has the same local specification at $S(K_1)$.
		So, we could count the number $\#\mathcal{C}(\Sigma)$ by the following method.
		Define $\mathcal{C}(G/G_1,\Sigma)$ to be the set of solutions to the $G/G_1$-extensions $K_1'/k$ unramified outside $S(k)$ such that $K'_{1,\mathfrak{p}}\cong K_{1,\mathfrak{p}}$ for each $\mathfrak{p}\in S(k)$.
		And for each $K'_1\in\mathcal{C}(G/G_1,\Sigma)$, define $\mathcal{C}(G_1,K'_1,\Sigma)$ to be the set of solutions to the $G_1$-extensions $K'/K'_1$ unramified outside $S(K_1')$ such that $K'_{\mathfrak{P}}\cong_{G_1} K_{\mathfrak{P}}$ for each $\mathfrak{P}\in S(K_1')$.
		Consider the sum
		\begin{equation*}
			\sum_{K'_1\in\mathcal{C}(G/G_1,\Sigma)}\#\mathcal{C}(G_1,K_1',\Sigma).
		\end{equation*}
		By induction assumption, we see that $\#\mathcal{C}(G_1,K_1',\Sigma)\leq\prod_{i=1}^{n-1}c_id_{K'_i}^{N_i}$. 
		It is independent of the choice of $K'_1$ in the sense that $d_{K'_i}=d_{K_i}$ for all $K'\in\mathcal{C}({\Sigma})$.
		The number of $K'_1\in\mathcal{C}(G/G_1,{\Sigma})$ is reduced to the number of abelian extensions over $k$ unramified outside $S$ with a fixed specification at $S(k)$, the case when $n=1$.
		So, we have
		\begin{equation*}
			\#\mathcal{C}({\Sigma})\leq c_0d_{K_0}^{N_0}\prod_{i=1}^{n-1}c_id_{K_i}^{N_i}=\prod_{i=0}^{n-1}c_id_{K_i}^{N_i}.
		\end{equation*}
	\end{proof}
	Generally speaking, this is just a ``trivial'' upper bound estimate for the restricted ramification of solvable fields, for it is just a generalization of the Minkowski bound for the class groups.
	However, under suitable conditions, we will see that this coarse upper bound is enough for us to prove some statistical results.
	
    \section{Solvable extension with a normal abelian subgroup}\label{section: Solvable extension with a normal abelian subgroup}
    In this section, let $G$ be a finite transitive permutation group with $\operatorname{Stab}_{G}(1)$ trivial that satisfies the following short exact sequence
    \begin{equation}\label{eqn: s.e.s. of the Galois group}
    	1\to N\to G\to H\to1
    \end{equation}
    where $N$ is an abelian $p$-group and $\gcd(\lvert N\rvert,\lvert H\rvert)=1$.
    Clearly, $N$ is a finite $H$-module, where the action of $H$ is given by the conjugation.
    To be precise, for each $h\in H$, let $\tilde{h}$ be a preimage in $G$ under the above sequence.
    Then for each $g\in N$, we have
    \begin{equation*}
    	h\cdot g:=\tilde{h}g\tilde{h}^{-1}.
    \end{equation*}
    It is routine to check that this is well-defined.
    Before we get into the statements on the arithmetic statistics, let us explain that this set-up actually includes more situations than it appears.
    \begin{proposition}\label{prop: solvable extension with a normal abelian subgroup}
    	Let $G$ be a finite solvable group with a tower
    	\begin{equation*}
    		\{e_{G}\}=G_{l}\triangleleft G_{l-1}\triangleleft\cdots\triangleleft G_{1}\triangleleft G_{0}=G
    	\end{equation*}
    	such that the quotient $G_{i}/G_{i+1}$ is abelian for all $0\leq i\leq l-1$.
    	If $G_{l-1}$ is normal in $G$ and $p$ is a prime such that $p\mid\lvert G_{l-1}\rvert$ and $p\nmid\lvert G/G_{l-1}\rvert$, then $G$ admits a unique (hence normal) Sylow $p$-subgroup $N$ included in $G_{l-1}$, and it fits into the short exact sequence as in (\ref{eqn: s.e.s. of the Galois group}).
    \end{proposition}
    \begin{proof}
    	Let $N$ be any Sylow $p$-subgroup.
    	The image of $N$ in $G/G_{l-1}$ must be a $p$-subgroup of the quotient group.
    	But according to our condition, the size of $G/G_{l-1}$ is coprime to $p$.
    	This shows that $N\subseteq G_{l-1}$.
    	Since $G_{l-1}$ is a subgroup of $G$, this means that $N$ must be a Sylow $p$-subgroup of $G_{l-1}$.
    	But according to the tower of $G$, the subgroup $G_{l-1}$ is abelian.
    	In other words, we have shown that any Sylow $p$-subgroup of $G$ must be included in $G_{l-1}$ and equal to the unique Sylow $p$-subgroup $N$ of $G_{l-1}$.
    	The rest of the statement just follows, and we are done for the proof.
    \end{proof}
    \begin{remark}
    	This statement shows that we could simply require that the finite group $G$ is solvable and it admits a normal abelian subgroup $N$ such that $p\mid\lvert N\rvert$ and $p\nmid\lvert G/N\rvert$, hence the title of this section.
    \end{remark}
    Let us restate and prove one of the main results, Theorem~\ref{thm: solvable group with abelian normal subgroup S1}, in the introduction section step-by-step.
    \begin{definition}\label{def: the set of G-specifications Phi}
    	Recall that $G$ is a transitive permutation group with $\operatorname{Stab}_{G}(1)=\{e_{G}\}$.
    	Let $e_{G}\notin\Omega\subseteq G$ be a subset closed under conjugation and invertible powering, that is, if $g_{1}\sim g_{2}$, then $g_{1}\in\Omega$ if and only if $g_{2}\in\Omega$.
    	Fix a number field $k$, and let $\mathcal{C}$ be a subset of $\mathcal{C}(G,k)$ with a generalized discriminant $C$.
    	\begin{enumerate}
    		\item For each positive integer $n$, define $T_{k}(n):=\{\mathfrak{p}\in\mathcal{P}_{k}\mid\mathfrak{p}\nmid n\infty\}$.
    		For a subset $T$ of $\mathcal{P}_{k}$, define $I_{T}^{+,\mu}$ to be the set of square-free integral ideals $\mathfrak{a}$ of $k$ such that $\mathfrak{p}\mid\mathfrak{a}\Rightarrow\mathfrak{p}\in T$.
    		\item Let $T:=T_{k}(\lvert G\rvert)$.
    		For each non-negative integer $\gamma$, define $\Phi(\Omega,\gamma)$ to be the set of $G$-specifications as follows:
    		\begin{equation*}
    			\Phi(\Omega,\gamma):=
    			\bigsqcup_{
    				\substack{
    					\mathfrak{a}\in I_{T}^{+,\mu} \\
    					\omega(\mathfrak{a})=\gamma
    				}
    			}\{\Sigma=(\Sigma_{\mathfrak{p}}:\rho_{\mathfrak{p}}(y_{\mathfrak{p}}^{t})\in\Omega)_{\mathfrak{p}\mid\mathfrak{a}}\},
    		\end{equation*}
    		where $\omega(\mathfrak{a})$ is the number of distinct primes dividing $\mathfrak{a}$, and $\rho_{\mathfrak{p}}:G_{k_{\mathfrak{p}}}\to G$ is the continuous homomorphism corresponding to $\Sigma_{\mathfrak{p}}$ (see also Definition~\ref{def: generalized discriminant of morphisms}).
    		\item For each non-negative integer $\gamma$, define
    		\begin{equation*}
    			\mathcal{C}_{\Omega}^{\gamma}:=\mathcal{C}(\Phi(\Omega,\gamma)).
    		\end{equation*}
    	\end{enumerate}		
    \end{definition}
    For the notation $\mathcal{C}_{\Omega}^{\gamma}$, we have a description from the aspect of the arithmetic.
    \begin{lemma}
    	Fix a number field $k$, and a subset $e_{G}\notin\Omega$ of $G$ that is closed under conjugation and invertible powering.
    	Let $\mathcal{C}:=\mathcal{C}(G,k)$.
    	A $G$-extension $K/k$ is contained in $\mathcal{C}_{\Omega}^{\gamma}$ if and only if there exists exactly $\gamma$ tamely ramified primes $\mathfrak{p}$ such that the inertia subgroup $I_{\mathfrak{p}}$ is generated by some element in $\Omega$ for each $1\leq i\leq\gamma$.
    \end{lemma}
    \begin{proof}
    	Note that $G$ has the property that $\operatorname{Stab}_{G}(1)$ is trivial.
    	So a $G$-extension $K/k$ is Galois with Galois group isomorphic to $G$, and $\mathcal{C}(G,k)$ is the set of Galois $G$-extensions.
    	Then the statement is just a translation of the definition, so we omit the details.
    \end{proof}	
    \begin{example}
    	Just for an example, let $G=\langle\sigma\mid\sigma^{2}=e_{G}\rangle\cong\mathbb{Z}/2\mathbb{Z}$, and $\Omega=\{\sigma\}$.
    	Define $\mathcal{C}:=\mathcal{C}(G,\mathbb{Q})$ to be the set of quadratic number fields.
    	And for each non-negative integer $\gamma$, the set $\mathcal{C}_{\Omega}^{\gamma}$ is the set of quadratic number fields $K/\mathbb{Q}$ such that there are exactly $\gamma$ odd ramified primes in the extension.
    \end{example}
    In the rest of this section, we define the following notations.
    \begin{definition}\label{def: notations of S5}
    	Recall that $G$ is a finite transitive permutation group with $\operatorname{Stab}_{G}(1)$ trivial and the structure~(\ref{eqn: s.e.s. of the Galois group}).
    	\begin{enumerate}
    		\item Clearly $G$ is solvable, and we have the following tower
    		\begin{equation*}
    			\{e_G\}=G_{l}\triangleleft N=G_{l-1}\triangleleft\cdots\triangleleft G_{0}=G
    		\end{equation*}
    		such that $G_{i+1}$ is normal in $G_i$ and $G_i/G_{i+1}$ is abelian for each integer $0\leq i\leq l-1$.
    		\item Define $\mathcal{C}:=\{K\in\mathcal{C}(G)\mid\mu(K)=\{\pm1\}\}$ where $\mu(K)$ is the group of roots of unity.
    		\item Let $c_G:G\to\mathbb{Z}_{\geq0}$ be a weight such that $\Omega:=c_{G}^{-1}(m)\subseteq N$.
    		Define
    		\begin{equation*}
    			m:=\min_{e_G\neq g\in G}\{c_G(g)\},\, M:=\max_{e_G\neq g\in G}\{c_G(g)\}.
    		\end{equation*}
    		Define $C:\mathcal{C}\to\mathbb{Z}_+$ to be a generalized discriminant with respect to $c_G$.
    	\end{enumerate}    	
    \end{definition}
    Applying the Theorem~\ref{thm: restricted ramification for solvable extensions} here, we could prove some upper bound estimate for the following statistical objects.
    \begin{proposition}\label{prop: moment for solvable extensions} 
    	Let $G,\mathcal{C},C$ be as in Definition~\ref{def: notations of S5}.
    	\begin{enumerate}
    		\item There exists some positive integer $a$ such that
    		\begin{equation*}
    			N_{\mathcal{C},C}(X)\ll X^{1/m}(\log X)^{a-1}.
    		\end{equation*}
    		\item For each $0<\epsilon<1$, we have
    		\begin{equation*}
    			N_{\mathcal{C}_{\Omega}^{0},C}(X)\ll X^{1/(m'-\epsilon)},
    		\end{equation*}
    		where $m':=\min\{c_{G}(g)\mid g\notin\{e_{G}\}\cup\Omega\}$.
    		And for each positive integer $\gamma$, we have
    		\begin{equation*}
    			N_{\mathcal{C}_{\Omega}^{\gamma},C}(X)\ll\frac{X^{1/m}}{\log X}(\log\log X)^{\gamma-1}.
    		\end{equation*}
    	\end{enumerate}
    \end{proposition}
    \begin{proof}
    	The proof of (1) and (2) follows the similar method, and (1) is on the field-counting while (2) is counting a specific subfamily.
    	But first, the condition that $\Omega\subseteq G_{l-1}$ puts more restrictions on the structure of $G$, that is, $\Omega$ generates an abelian normal subgroup of $G$ included in $G_{l-1}$.
    	So, we may assume without loss of generality that $G_{l-1}=\langle\Omega\rangle$.
    	
    	As in the Theorem~\ref{thm: restricted ramification for solvable extensions}, for each $0\leq i\leq l-1$, let $c_i$ and $N_i$ be the constants from Lemma~\ref{lemma: restricted ramification for sovlable extensions} depending only on $n_{i}:=\lvert G/G_i\rvert$ (the base field $k$ is $\mathbb{Q}$ in this case) and $G_{i}/G_{i+1}$ such that for each number field $[L:\mathbb{Q}]=n_{i}$, we have
    	\begin{equation*}
    		\lvert\operatorname{Hom}(\operatorname{Cl}_{L},G_{i}/G_{i+1})\rvert\leq c_{i}d_{L}^{N_{i}}.
    	\end{equation*}
    	Note that $d_{K_{i}}\mid d_{K_{i+1}}$ for each integer $0\leq i\leq l-1$.
    	Let $c=\prod_{i=0}^{l-1} c_{i}$ and $N$ be a constant such that
    	\begin{equation}\label{eqn: moment for solvable extensions 1}
    		\prod_{i=0}^{l-1} c_id_{K_i}^{N_i}\leq c(\sqrt{d_{K_{l-1}}})^N,
    	\end{equation}
    	where $\sqrt{d_{K}}$ is the radical of the absolute discriminant.
    	In particular, for a subset $S$ of $\mathcal{P}$ containing the ones at $\lvert G\rvert\infty$, for each $G$-specification $\Sigma$ at $S$, we have
    	\begin{equation*}
    		\#\mathcal{C}(\Sigma)\leq c(\sqrt{d_{\Sigma_{l-1}}})^{N},
    	\end{equation*}
    	where $d_{\Sigma_{l-1}}=\prod_{p}d_{\Sigma_{l-1,p}}$, the product of the local discriminant.
    	Let $M:=\max_{g\in G}c_{G}(g)$, and $a$ be a positive integer such that for each $p\in\mathcal{P}$, we have
    	\begin{equation*}
    		\#\{\Sigma_{\mathfrak{p}}\text{ is a }G\text{-specification at }p\}\leq a.
    	\end{equation*}	
    	Claim: under the condition that for all $g\in G$ if $c_{G}(g)>m$ then $c_{G}(g)-N>m$, the series
    	\begin{equation*}
    		D(s):=c\prod_{p}(1+ap^{-ms}+ap^{N}\sum_{i=N+m+1}^{M}p^{-is})=\sum_{n=1}^{\infty}a_{n}n^{-s}
    	\end{equation*}
    	satisfies the property that for each $X>0$, we have
    	\begin{equation*}
    		N_{\mathcal{C},C}(X)\leq\sum_{n<X}a_{n}.
    	\end{equation*}
    	If the Claim is true, then by the author~\cite[Corollary 5.4]{wang2026invariant}, we know immediately that 
    	\begin{equation*}
    		N_{\mathcal{C},C}(X)\leq\sum_{n<X}a_{n}\ll X^{1/m}(\log X)^{a-1}.
    	\end{equation*}
    	Let us prove the Claim by comparing the coefficient directly.
    	Let $d$ be square-free, and $f$ be a positive integer such that if a prime $p\mid f$ then $p^{N+m+1}\mid f$.
    	Clearly for each $K/\mathbb{Q}\in\mathcal{C}$, its corresponding value $C(K)$ under the counting function is of the form $C(K)=d^{m}f$.
    	This shows that if $a_{n}=0$, then $\#\mathcal{C}(n)=0$, because every positive integer $n$ of the form $df$ corresponds to a nonzero coefficient $a_{n}$.
    	Conversely, if $a_{n}>0$, then there exists some $K/\mathbb{Q}$ such that $C(K)=n$.
    	In this case, we see that if $p\mid d_{K_{n-1}}$, then $p^{N+m+1}\mid C(K)$, hence $\sqrt{d_{K_{n-1}}}\mid\sqrt{f}$ where $\sqrt{}$ means the radical.
    	This implies that
    	\begin{align*}
    		a_{n}=&a^{\omega(n)}\cdot c(\sqrt{f})^{N} \\
    		\geq &\sum_{\Sigma:C(\Sigma)=n}c(\sqrt{f})^{N} \\
    		\geq &\sum_{\Sigma:C(\Sigma)=n}c(\sqrt{d_{\Sigma_{n-1}}})^{N} \\
    		\geq &\sum_{\Sigma:C(\Sigma)=n}\#\mathcal{C}(\Sigma),
    	\end{align*}
    	where $\omega(n)$ is the number of distinct prime factors of $n$.
    	So, for each $n\in\mathbb{Z}_{+}$, we have that
    	\begin{equation*}
    		a_{n}\geq\#\mathcal{C}(n),
    	\end{equation*}
    	hence the claim.
    	For general case when $c_{G}(g)-N\leq m$ if $g\notin\Omega\cup\{e_{G}\}$, let us adopt the similar method.
    	Let $r$ be a large enough positive integer such that $rc_{G}(g)-N>rm$ for all $g\notin\Omega\cup\{e_{G}\}$.
    	Then the generalized discriminant $C^{r}(K)=C(K)^{r}$ satisfies the condition the Claim.
    	So, there exists some positive integer $a$ such that
    	\begin{align*}
    		N_{\mathcal{C},C^{r}}(X)\ll& X^{\frac{1}{mr}}(\log X)^{a-1} \\
    		\Rightarrow N_{\mathcal{C},C}(X)=N_{\mathcal{C},C^{r}}(X^{r})\ll&X^{r\frac{1}{mr}}(\log X^{r})^{a-1}\ll X^{1/m}(\log X)^{a-1}.
    	\end{align*}
    	Finally we move on to the proof of (3), which is the estimate for $N_{\mathcal{C}_{\Omega}^{\gamma},C}(X)$.
    	Let $T:=T_{\mathbb{Q}}(\lvert G\rvert)$, that is, $T=\{p\in\mathcal{P}: p\nmid\lvert G\rvert\infty\}$.
    	
    	Claim: under the condition that for all $g\in G$ if $c_{G}(g)>m$ then $c_{G}(g)-N>m$, the series
    	\begin{align*}
    		E_{\gamma}(s):=& c\Bigl(\sum_{\Sigma_{\infty}}1\Bigr)
    		\prod_{p\mid\lvert G\rvert}\Bigl(\sum_{\Sigma_{p}}C_{p}(\Sigma_{p})^{-s}\Bigr)
    		\prod_{p\in T}(1+ap^{N}\sum_{i=N+m+1}^{M}p^{-is})
    		\Bigl(\sum_{
    			\substack{
    				d\in I_{T}^{+,\mu} \\
    				\omega(d)=\gamma
    			}
    		}a^{\gamma}d^{-ms}\Bigr) \\
    		=& \sum_{n=1}^{\infty}b_{\gamma,n}n^{-s}
    	\end{align*}
    	satisfies the property that
    	\begin{equation*}
    		N_{\mathcal{C}_{\Omega}^{\gamma},C}(X)\leq\sum_{n<X}b_{\gamma,n}.
    	\end{equation*}
    	If the Claim is true, then by combining the author~\cite[5.1, 5.2, 5.4]{wang2026invariant}, when $\gamma\geq1$, we have
    	\begin{equation*}
    		N_{\mathcal{C}_{\Omega}^{0},C}(X)\ll X^{1/(m'-N)}\quad\text{and}\quad N_{\mathcal{C}_{\Omega}^{\gamma},C}(X)\ll\frac{X^{1/m}}{\log X}(\log\log X)^{\gamma-1}.
    	\end{equation*}
    	Then we use the same trick as above for the general case and also the statement for $\gamma=0$.
    	To be precise, for each $0<\epsilon<1$, let $r$ be a large enough integer such that $rc_{G}(g)-N>rm$ and $N/r<\epsilon$.
    	Then for $\gamma\geq1$, we have
    	\begin{align*}
    		N_{\mathcal{C}_{\Omega}^{\gamma},C}(X)=& N_{\mathcal{C}_{\Omega}^{\gamma},C^{r}}(X^{r}) \\
    		\ll & \frac{X^{r/rm}}{\log X}(\log\log X)^{\gamma-1} \\
    		=& \frac{X^{1/m}}{\log X}(\log\log X)^{\gamma-1}.
    	\end{align*}
    	And for $\gamma=0$, we have
    	\begin{equation*}
    		N_{\mathcal{C}_{\Omega}^{0},C}(X)=N_{\mathcal{C}_{\Omega}^{0},C^{r}}(X^{r})\ll X^{\frac{r}{rm'-N}}=X^{\frac{1}{m'-N/r}}\leq X^{\frac{1}{m'-\epsilon}}.
    	\end{equation*}
    	It then suffices to show the Claim.
    	We consider the similar estimation.
    	For each positive integer $n$, the set $\mathcal{C}_{\Omega}^{\gamma}(n)$ is nonempty only if $\#\{p\in T:p^{m}\| n\}=\gamma$ and for each $p\mid n$ we have $p^{m}\mid n$.
    	Let us write $\prod_{p\in T:p^{m}\| n}p=d$.
    	Under such condition, we have
    	\begin{align*}
    		\#\mathcal{C}_{\Omega}^{\gamma}(n)\leq &\sum_{\Sigma:C(\Sigma)=n} c(\sqrt{d_{\Sigma_{l-1}}})^{N} \\
    		\leq & \sum_{\Sigma:C(\Sigma)=n}c(\sqrt{n/d^{m}})^{N} \\
    		\leq & c(\sqrt{n/d^{m}})^{N}\#\{\Sigma_{\infty}\}\prod_{p\mid\lvert G\rvert,p\mid n}\#\{\Sigma_{p}: C(\Sigma_{p})\|n\}\prod_{p\in T,p\mid n}\#\{\Sigma_{p}\} \\
    		\leq & c(\sqrt{n/d^{m}})^{N}\#\{\Sigma_{\infty}\}\prod_{p\mid\lvert G\rvert,p\mid n}\#\{\Sigma_{p}: C(\Sigma_{p})\|n\}\prod_{p\in T,p\mid n}a \\
    		=& b_{\gamma,n}.
    	\end{align*}
    	So the Claim is true, hence the proof of (2).
    \end{proof}
    Now we could prove the main result on the field-counting as follows.
    \begin{theorem}\label{thm: proof of the hypothesis S5}
    	Let $G,\mathcal{C},C$ be as in Definition~\ref{def: notations of S5}.
    	Then for each non-negative integer $\gamma$, we have
    	\begin{equation*}
    		N_{\mathcal{C}_{\Omega}^{\gamma},C}(X)=o(N_{\mathcal{C}_{\Omega}^{\gamma+1},C}(X)).
    	\end{equation*}
    	In particular, if $\gamma>0$, then
    	\begin{equation*}
    		N_{\mathcal{C}_{\Omega}^{\gamma},C}(X)\asymp\frac{X^{1/m}}{\log X}(\log\log X)^{\gamma-1}.
    	\end{equation*}
    \end{theorem}
    \begin{proof}
    	By the author~\cite[Lemma 7.11]{wang2026invariant}, for each $\gamma\in\mathbb{Z}_{+}$, we have
    	\begin{equation*}
    		N_{\mathcal{C}_{\Omega}^{\gamma},C}(X)\gg \frac{{X}^{1/m}}{\log X}(\log\log X)^{\gamma-1}.
    	\end{equation*}
    	Apply the Proposition~\ref{prop: moment for solvable extensions} here, we immediately obtain the desired result.
    \end{proof}
    As for the moment of class groups, we have the following result.
    \begin{theorem}\label{thm: zero-probability and infinite moment S5}
    	Let $G,\mathcal{C},C$ be as in Definition~\ref{def: notations of S5}.
    	If $N$ is a $p$-group, then for each non-negative $r\in\mathbb{Z}$, we have
    	\begin{equation*}
    		\mathbb{P}_{\mathcal{C},C}(\operatorname{rk}_{p}\operatorname{Cl}_{K}\leq r)=0\quad\text{and}\quad \mathbb{E}_{\mathcal{C},C}(\lvert\operatorname{Hom}(\operatorname{Cl}_{K},\mathbb{Z}/p\mathbb{Z})\rvert)=+\infty.
    	\end{equation*}
    \end{theorem}
    \begin{proof}
    	This is a direct corollary of \cite[Theorem 3.14]{wang2026invariant}.
    	For each rational prime $q$, let $e(q,L/\mathbb{Q})$ be its ramification index in the Galois extension $L/\mathbb{Q}\in\mathcal{C}$.
    	According to the author~\cite[Theorem 3.1]{wang2026invariant}, we have 
    	\begin{align*}
    		\operatorname{rk}_{p}\operatorname{Cl}_{L}\geq&\#\{q\in\mathcal{P}\mid e(q,L/\mathbb{Q})\equiv 0 \bmod{p}\}-[L:\mathbb{Q}]^{2} \\
    		\geq&\#\{q\nmid\lvert G\rvert\infty:\rho_{q}(y_{q}^{t})\in\Omega\}-[L:\mathbb{Q}]^{2}.
    	\end{align*}
    	where $\rho_{q}:G_{\mathbb{Q}_{q}}\to G$ is the local map corresponding to $L_{p}$.
    	Let $T:=T_{\mathbb{Q}}(\lvert G\rvert)=\{p\in\mathcal{P}:p\nmid\lvert G\rvert\infty\}$.
    	We then see that $((\mathcal{C},C),\Omega,T)$ satisfies the condition of \cite[Theorem 3.14]{wang2026invariant}.
    	That is, there exists some constant $c\geq0$, for each $L\in\mathcal{C}$, we have
    	\begin{equation*}
    		\operatorname{rk}_{p}\operatorname{Cl}_{L}\geq\#\{q\in T:\rho_{q}(y_{q}^{t})\in\Omega\}-c,
    	\end{equation*} 
    	and for each non-negative integer $\gamma$, we have that
    	\begin{equation*}
    		N_{\mathcal{C}_{\Omega}^{\gamma},C}(X)=o(N_{\mathcal{C}_{\Omega}^{\gamma+1},C}(X)).
    	\end{equation*}
    	Then \cite[Theorem 3.14]{wang2026invariant} implies immediately the zero-probability and the infinite $\mathbb{Z}/p\mathbb{Z}$-moment.
    \end{proof}
    \begin{remark}
    	The Theorem~\ref{thm: solvable group with abelian normal subgroup S1} is proved directly by the combination of the Proposition~\ref{prop: solvable extension with a normal abelian subgroup} and the above Theorem~\ref{thm: zero-probability and infinite moment S5}.
    	So, we have done for one of the main results of this paper.
    \end{remark}
    Then let us show some applications of the Theorem~\ref{thm: zero-probability and infinite moment S5}.
    \begin{lemma}
    	Let $K/\mathbb{Q}$ be a Galois $\Gamma$-extension, where $\Gamma$ is a finite group.
    	Let $p$ be a finite rational prime such that $p\nmid\lvert\Gamma\rvert$.
    	If $e=e(p,K/\mathbb{Q})$ is the ramification index, and $f=f(p,K/\mathbb{Q})$ is the inertia degree, then
    	\begin{equation*}
    		v_{p}(d_{K})=\lvert\Gamma\rvert(1-e^{-1}),%
    	\end{equation*}
    	where $v_{p}:\mathbb{Z}\to\mathbb{N}\cup\{\infty\}$ is the normalized valuation at $p$ and $d_{K}$ is the (absolute) discriminant of $K$.
    \end{lemma}
    \begin{proof}
    	By Neukirch~\cite[(3.2.6)]{neukirch2013algebraic}, if $\mathfrak{p}$ is a finite prime of $K$ that is tamely ramified in $K/\mathbb{Q}$, then we have
    	\begin{equation*}
    		\mathfrak{D}_{\mathfrak{p}}=\mathfrak{D}(K_{\mathfrak{p}}/\mathbb{Q}_{p})=\mathfrak{p}^{e-1},%
    	\end{equation*}
    	where $\mathfrak{D}$ is the different and $e$ is the ramification index of $\mathfrak{p}$ which is the same as the ramification index of $p$.
    	The basic relation between the different and the discriminant is given by
    	\begin{equation*}
    		d_{\mathfrak{p}}=d(K_{\mathfrak{p}}/\mathbb{Q}_{p})=\mathfrak{ND}_{\mathfrak{p}},%
    	\end{equation*}
    	where $\mathfrak{N}=\operatorname{Nm}_{K/\mathbb{Q}}$.
    	See Neukirch~\cite[(3.2.9), (3.2.11)]{neukirch2013algebraic} for example.
    	Therefore, if $p$ has the splitting $p\mathcal{O}_{K}=(\mathfrak{p}_{1}\cdots\mathfrak{p}_{l})^{e}$ in $K/\mathbb{Q}$, then we have
    	\begin{align*}
    		d(K_{p}/\mathbb{Q}_{p})=&\prod_{i=1}^{l}d_{\mathfrak{p}_{i}}\\%
    		=&\prod_{i=1}^{l}\mathfrak{ND}_{\mathfrak{p}_{i}}(K/\mathbb{Q})\\%
    		=&(p^{f})^{(e-1)l}.%
    	\end{align*}
    	Recall the fundamental identity: $efl=\lvert\Gamma\rvert$.
    	So, we have
    	\begin{equation*}
    		fl(e-1)=\lvert\Gamma\rvert(1-e^{-1}),%
    	\end{equation*}
    	hence the statement.
    \end{proof}
    This computation shows immediately the following.
    \begin{corollary}
    	Fix a finite group $\Gamma$.
    	If there exists at least one Galois $\Gamma$-extension $K/\mathbb{Q}$, then as a generalized discriminant, the (absolute) discriminant corresponds to the weight $c_{\Gamma}:\Gamma\to\mathbb{Z}_{\geq0}$ such that for each $g\neq e_{\Gamma}$ we have
    	\begin{equation*}
    		c_{\Gamma}(g)=\lvert\Gamma\rvert(1-r_{g}^{-1}),%
    	\end{equation*}
    	where $r_{g}$ is the order of $g$ in $\Gamma$.
    \end{corollary}
    Now that we have established a simple relation between the absolute discriminant and the weight function, the Theorem~\ref{thm: zero-probability and infinite moment S5} could be applied to the following cases.
    \begin{proposition}\label{prop: applications and examples S5}
    	Keep the notations as in (\ref{eqn: s.e.s. of the Galois group}) and the Definition~\ref{def: notations of S5}.
    	If $p$ is the smallest prime dividing $\lvert G\rvert=\lvert N\rvert\times\lvert H\rvert$, then for each non-negative integer $\gamma$, we have
    	\begin{equation*}
    		N_{\mathcal{C}_{\Omega}^{\gamma},d}(X)=o(N_{\mathcal{C}_{\Omega}^{\gamma+1},d}(X)),%
    	\end{equation*}
    	where $d$ is the absolute discriminant.
    	In particular, if $\gamma>0$, then
    	\begin{equation*}
    		N_{\mathcal{C}_{\Omega}^{\gamma},d}(X)\asymp\frac{X^{1/m}}{\log X}(\log\log X)^{\gamma-1},%
    	\end{equation*}
    	where $m=\lvert G\rvert(1-p^{-1})$.
    	Moreover, we have
    	\begin{equation*}
    		\forall r\geq0,\,\mathbb{P}_{\mathcal{C},d}(\operatorname{rk}_{p}\operatorname{Cl}_{K}\leq r)=0%
    		\quad\text{and}\quad%
    		\mathbb{E}_{\mathcal{C},d}(\lvert\operatorname{Hom}(\operatorname{Cl}_{K},\mathbb{Z}/p\mathbb{Z})\rvert)=+\infty.
    	\end{equation*}
    \end{proposition}
    \begin{proof}
    	It suffices to check that the absolute discriminant $d$ satisfies the condition of this section~(\ref{def: notations of S5}).
    	Since $p$ is the smallest prime that divides $\lvert G\rvert$, we see that if $g\in N$ has order $p$, then
    	\begin{equation*}
    		c_{G}(g)=\lvert G\rvert(1-p^{-1}),%
    	\end{equation*}
    	where $c_{G}$ is the weight that corresponds to the absolute discriminant.
    	Clearly, for all nontrivial elements $g\in G$, we have
    	\begin{equation*}
    		c_{G}(g)=\lvert G\rvert(1-r_{g}^{-1})\geq\lvert G\rvert(1-p^{-1}).
    	\end{equation*}
    	So it is true that $m=\lvert G\rvert(1-p^{-1})$, and only the elements of order $p$ could reach this minimum $m$.
    	According to our assumption on the group structure, this means that $\Omega=c_{G}^{-1}(m)\subseteq N$.
    	So, in this case, all the conditions of the Theorem~\ref{thm: proof of the hypothesis S5} and~\ref{thm: zero-probability and infinite moment S5} are satisfied.
    	And we are done.
    \end{proof}
    \begin{example}%
    	For an example, the author~\cite{wang2026invariant} has discussed $(6,A_{4})$-fields ordered by the absolute discriminant in details.
    	In particular, the group $A_{4}$ admits the following short exact sequence
    	\begin{equation*}
    		1\to V_{4}\to A_{4}\to \mathbb{Z}/3\mathbb{Z}\to1.%
    	\end{equation*}  
    	Also, this result is interesting in the following sense.
    	If $L/\mathbb{Q}$ is a Galois $A_{4}$-field, then we see that $L$ has a unique Galois cubic subfield $K$.
    	If we believe the (weak) Cohen-Lenstra Heuristics, or check the statistics, then for Galois cubic fields, the $\mathbb{Z}/2\mathbb{Z}$-moment of the class groups should be finite.
    	This indicates that the infinite $\mathbb{Z}/2\mathbb{Z}$-moment of the class groups for the Galois $A_{4}$-fields (or equivalently the $(6,A_{4})$-fields) comes from the $V_{4}$-extension $L/K$.
    	
    	For another example, or rather, a method of constructing concrete examples, we introduce the wreath product.
    	If $H$ has a permutation action on a finite set $S$, then for any abelian group $A$, define
    	\begin{equation*}
    		A\wr H:=A^{S}\rtimes H.%
    	\end{equation*}
    	Write $\bar{a}:=(a_{s})_{s\in S}\in A^{S}$.
    	The group law for the subgroup $A^{S}$ is just pointwise.
    	For each $(\bar{a}_{1},h_{1}),(\bar{a}_{2},h_{2})\in A\wr H$, the group law is given by
    	\begin{equation*}%
    		(\bar{a}_{1},h_{1})\cdot(\bar{a}_{2},h_{2}):=(\bar{a}_{1}\cdot(h_{1}\cdot\bar{a}_{2}),h_{1}h_{2}),%
    	\end{equation*}%
    	where the action of $H$ on $A$ is given by the permutation on the coordinates $s\in S$.
    	That is, if we write $\bar{b}=h\cdot\bar{a}$, then $b_{s}=a_{h^{-1}s}$, and we could abbreviate it as:
    	\begin{equation*}%
    		h\cdot\bar{a}:=(a_{h^{-1}s}).%
    	\end{equation*}%
    	Now let $p$ be a fixed rational prime.
    	If $H$ is a finite solvable group with a permutation action on $S$ such that its size coprime to all integers $\leq p$, then for any finite abelian $p$-group $A$, define $G:=A\wr H$.
    	It is routine to check that this construction of $G$ will produce a set-up that satisfies all the conditions in the Definition~\ref{def: notations of S5} and the Proposition~\ref{prop: applications and examples S5}.
    	And the set $\mathcal{C}$ of fields will satisfy the statistical properties in the Proposition~\ref{prop: applications and examples S5} directly.
    	
    	Following this method, let us give a concrete example.
    	Define
    	\begin{equation*}%
    		H:=\Biggl\{%
    		\begin{pmatrix}%
    			1&a&b\\%
    			0&1&c\\%
    			0&0&1
    		\end{pmatrix}%
    		\mid a,b,c\in\mathbb{F}_{5}
    		\Biggr\}.%
    	\end{equation*}%
    	It is clear that the upper triangle matrices with $1$s on the diagonal form a subgroup of $\operatorname{SL}_{3}(\mathbb{F}_{5})$ under the matrix multiplication.
    	Let $H$ act on itself by left multiplication to obtain the permutation action, and $A$ be the finite abelian $3$-group $\mathbb{Z}/9\mathbb{Z}$.
    	Then define $G:=A\wr H$, which is isomorphic to $A^{125}\times H$ as a set, and it satisfies the required properties from the Proposition~\ref{prop: applications and examples S5}.
    	And this is an example where the prime $p$ is taken to be odd, the group $A$ is not an elementary $p$-group, and the quotient group $H$ is not commutative.
    \end{example}%
    
    \section{Product of ramified primes}\label{section: product of ramified primes}
    In this section, we mainly deal with the case when the generalized discriminant is equivalent to the product of ramified primes.
    \begin{definition}\label{def: notations of S6}
    	Let $G$ be a finite permutation group with $\operatorname{Stab}_{G}(1)$ trivial that satisfies (\ref{eqn: s.e.s. of the Galois group}).
    	Define $\Omega:=N\backslash\{e_{G}\}$.
    	Define $\mathcal{C}:=\{L\in\mathcal{C}(G)\mid \mu(L)=\{\pm1\}\}$ where $\mu(L)$ is the group of roots of unity, and define $\mathcal{D}:=\{K=L^{N}\mid L\in\mathcal{C}\}$.
    \end{definition}
    We define a ``relative product of ramified primes'' specifically for the short exact sequence (\ref{eqn: s.e.s. of the Galois group}).
    The main idea is that if $L/\mathbb{Q}$ is a Galois $G$-field with $K=L^{N}$, then
    \begin{equation*}
    	C(L)=C(K)C(L/K).
    \end{equation*}
    However, following the language of the generalized discriminant requires some detailed study for the tower of subgroups/extensions.
    Let us construct $C(L/K)$ step-by-step.
    First of all, $L/K$ is Galois with Galois group isomorphic to $N$.
    But the field $L$ is more than just being Galois over $K$, it is also Galois over $\mathbb{Q}$.
    So, we have the following.
    \begin{proposition}\cite[Proposition 7.6]{wang2026invariant}
    	Fix a Galois $H$-field $K$.
    	Note that the id{\`e}les class group $C_{K}$ is a topological $H$-module.
    	There is a one-to-one correspondence between the following two sets
    	\begin{equation*}
    		\operatorname{Sur}_{H}(\operatorname{C}_{K},N)\leftrightarrow\{L\in\mathcal{C}(G)\mid K\subseteq L\},
    	\end{equation*}
    	where $\operatorname{Sur}_{H}$ means the surjective continuous group homomorphisms that is $H$-equivariant.
    \end{proposition}
    Note that since $N$ is abelian, a continuous map $G_{K_{\mathfrak{p}}}\to N$ must factor through $K_{\mathfrak{p}}^{*}\to N$ by local Class Field Theory, where $\mathfrak{p}$ is a prime of $K$.
    In particular, when $\mathfrak{p}\nmid\lvert N\rvert$, the maximal tamely ramified abelian extension $K_{\mathfrak{p}}^{ab,t}/K$ has Galois group isomorphic to $\hat{\mathbb{Z}}\times\mu_{q-1}$, where $q$ is the size of the residue field of $K_{\mathfrak{p}}$.
    The group $\hat{\mathbb{Z}}$ corresponds to the maximal unramified extension.
    So, there is a natural quotient map $G_{K_{\mathfrak{p}}}^t\to G_{K_{\mathfrak{p}}}^{ab,t}$, and we could take the ``tame inertia generator'' $y_{\mathfrak{p}}^{t}$ as some generator of $\mu_{q-1}$.
    That is, in the context of the local Class Field Theory, for each continuous map $\chi:K_{\mathfrak{p}}^{*}\to N$ with $\mathfrak{p}\nmid\lvert N\rvert\infty$, the tame inertia generator $y_{\mathfrak{p}}^{t}$ is some generator of $\mu_{q-1}\subseteq K_{\mathfrak{p}}^{*}$.
    \begin{definition}\label{def: relative product of ramified primes}
    	Define the weight $c_{G}:G\to\mathbb{Z}_{\geq0}$ by $c_{G}(g)=1$ for all $g\neq e_{G}$, and let $C:\mathcal{C}\to\mathbb{Z}_{+}$ be the generalized discriminant associated to $c_{G}$ given by the formula
    	\begin{equation*}
    		C(L):=\prod_{p\nmid\lvert G\rvert\infty}C_{p}(L_{p})\cdot\prod_{p\mid\lvert G\rvert}d_{L_{p}}.
    	\end{equation*}
    	\begin{enumerate}
    		\item Let $K\in\mathcal{D}$.
    		For each prime $\mathfrak{p}\nmid\lvert G\rvert d_{K}$ of $K$, for each continuous homomorphism $\chi_{\mathfrak{p}}:K_{\mathfrak{p}}^{*}\to N$ that is $H$-equivariant, define
    		\begin{equation*}
    			C_{\mathfrak{p}}(\chi_{\mathfrak{p}})=p^{c_{N}(\chi_{p}(y_{p}^{t}))/n_{p}},
    		\end{equation*}
    		where $y_{p}^{t}$ is taken to be any generator of $\mu(K_{\mathfrak{p}})$, and $n_{p}$ is the number of distinct primes of $K$ above $p$.
    		Define
    		\begin{equation*}
    			C'(L/K):=\prod_{\mathfrak{p}\nmid d_{K}\lvert G\rvert}C_{\mathfrak{p}}(\chi_{\mathfrak{p}})
    			\cdot\prod_{\mathfrak{p}\nmid d_{K},\mathfrak{p}\mid\lvert G\rvert}\mathfrak{Nd}_{L_{\mathfrak{p}}/K_{\mathfrak{p}}},
    		\end{equation*}
    		where $\chi:\operatorname{C}_{K}\to N$ is the continuous surjective map corresponding to $L/K$.
    		The notation $C'$ means that it is different from the generalized discriminant $C(L/K)=\prod_{\mathfrak{p}}C_{\mathfrak{p}}(L_{\mathfrak{p}}/K_{\mathfrak{p}})$ as in the introduction part.
    		\item Let $c_{H}:H\to\mathbb{Z}_{\geq0}$ be the weight defined by $c_{H}(h)=1$ for all $h\neq e_{H}$.
    		For each $p\nmid\lvert G\rvert\infty$, define
    		\begin{equation*}
    			\bar{C}_{p}(K_{p}):=p^{c_{H}(\bar{\rho}_{p}(y_{p}^{t}))},
    		\end{equation*}
    		where $\bar{\rho}_{p}:G_{\mathbb{Q}_{p}^{t}}\to H$ is the continuous homomorphism that corresponds to $K_{p}$.
    		For each $K\in\mathcal{D}$, define
    		\begin{equation*}
    			\bar{C}(K):=\prod_{p\nmid\lvert G\rvert\infty}\bar{C}_{p}(K_{p})\cdot\prod_{p\mid\lvert G\rvert}d^{[L:K]}_{K_{p}}.
    		\end{equation*}
    	\end{enumerate}
    \end{definition}
    Let us show that this definition really admits the desired property.
    \begin{lemma}
    	For each $L\in\mathcal{C}$, let $K:=L^{N}$.
    	We have
    	\begin{equation*}
    		C(L)=\bar{C}(K)C'(L/K).
    	\end{equation*}
    \end{lemma}
    \begin{proof}%
    	This just follows from the computation.
    	Fix a field $L\in\mathcal{C}$ with $K=L^{N}$.
    	For each prime $p$, let $\rho_{p}:G_{\mathbb{Q}_{p}}\to G$ be the continuous map corresponding to $L_{p}$.
    	The quotient map $G\to H$ naturally induces a continuous map $\bar{\rho}_{p}:G_{\mathbb{Q}_{p}}\to H$ that corresponds to $K_{p}$.
    	For a prime $p\nmid\lvert G\rvert\infty$, we have that $p\mid C(L)\iff \rho_{p}(y_{p}^{t})$ is nontrivial.
    	Therefore, either $\bar{\rho}_{p}(y_{p}^{t})$ is nontrivial, or $\rho_{p}(y_{p}^{t})\in N\backslash\{e_{G}\}$.
    	Let $\chi:\operatorname{C}_{K}\to N$ be the surjective continuous $H$-morphism that corresponds to $L/K$.
    	Using these basic observations, we have
    	\begin{equation*}%
    		\begin{aligned}%
    			C(L)=& \prod_{p\mid\lvert G\rvert}d_{L_{p}}\cdot\prod_{p\nmid\lvert G\rvert\infty}C_{p}(L_{p}) \\%
    			=& \prod_{p\mid\lvert G\rvert}d_{K_{p}}^{[L:K]}\prod_{\mathfrak{p}\mid p}\mathfrak{Nd}_{L_{\mathfrak{p}}/K_{\mathfrak{p}}}%
    			\cdot\prod_{p\nmid\lvert G\rvert\infty}p^{c_{G}(\rho_{p}(y_{p}^{t}))} \\%
    			=& \prod_{p\mid\lvert G\rvert}d_{K_{p}}^{[L:K]}\prod_{\mathfrak{p}\mid p}\mathfrak{Nd}_{L_{\mathfrak{p}}/K_{\mathfrak{p}}}%
    			\cdot\prod_{%
    				\substack{%
    					p\nmid\lvert G\rvert\infty \\%
    					\bar{\rho}_{p}(y_{p}^{t})\neq e_{H}%
    				}%
    			}p%
    			\cdot\prod_{%
    				\substack{%
    					p\nmid\lvert G\rvert\infty \\%
    					\rho_{p}(y_{p}^{t})\in N\backslash\{e_{G}\}%
    				}%
    			}p \\%
    			=& \Bigl(\prod_{p\mid\lvert G\rvert}d_{K_{p}}^{[L:K]}\prod_{p\nmid\lvert G\rvert\infty}\bar{C}_{p}(\bar{\rho}_{p})\Bigr)%
    			\cdot \Bigl(\prod_{\mathfrak{p}\mid\lvert G\rvert}\mathfrak{Nd}_{L_{\mathfrak{p}}/K_{\mathfrak{p}}}\prod_{\mathfrak{p}\nmid\lvert G\rvert\infty}C_{\mathfrak{p}}(\chi_{\mathfrak{p}})\Bigr) \\%
    			=& \bar{C}(K)\cdot C'(L/K)%
    		\end{aligned}
    	\end{equation*}%
    \end{proof}%
    In the rest of this section, we use this product of ramified primes $C$ as the counting function for $\mathcal{C}$.
    Moreover, we need an assumption on the statistics for $\mathcal{D}$.
    \begin{hypothesis}\label{hypthesis}
    	Keep the notations in the Definition~\ref{def: notations of S6}, there exists some non-negative integer $b$ such that
    	\begin{equation*}
    		\sum_{
    			\substack{
    				K\in\mathcal{D}\\
    				\bar{C}(K)<X
    			}
    		}\lvert\operatorname{Hom}_{H}(\operatorname{Cl}_{K},N)\rvert\ll X(\log X)^{b}
    	\end{equation*}
    	as $X\to\infty$.
    \end{hypothesis}
    \begin{example}
    	For example, let $\mathcal{C}$ be the set of Galois $S_{3}$-fields excluding the ones with $\zeta_{3}\in L$.
    	Using the result by Davenport-Heilbronn~\cite{Davenport1971Cubic} (see also Bhargava~\cite{bhargava2013davenport}), we know that
    	\begin{equation*}
    		\sum_{
    			\substack{
    				[K:\mathbb{Q}]=2 \\
    				K\neq\mathbb{Q}(\mu_{3}),d_{K}<X
    			}
    		}\lvert\operatorname{Hom}(\operatorname{Cl}_{K},\mathbb{Z}/3\mathbb{Z})\rvert\asymp X
    	\end{equation*}
    	as $X\to\infty$.
    	Roughly speaking, the Hypothesis is true for $\mathcal{C}$ with $b=0$.
    \end{example}
    Before proving the main result of this section, we need one more lemma as a technical tool of the estimate.
    \begin{lemma}\label{lemma: upper bound estimate S6}
    	Keep the notations in the Definition~\ref{def: notations of S6}.
    	For each positive integer $\gamma$, there exists some constant $c_{\gamma}>0$, for each $K\in\mathcal{D}$, and for each $X>0$, we have
    	\begin{equation*}
    		\#\{L\in\mathcal{C}_{\Omega}^{\gamma}\mid K\subseteq L,\, C'(L/K)<X\}\leq c_{\gamma}\lvert\operatorname{Hom}_{H}(\operatorname{Cl}_{K},N)\rvert\max\{1,\frac{X}{\log X}(\log\log X)^{\gamma-1}\}.
    	\end{equation*}
    \end{lemma}
    \begin{proof}
    	Let $a$ be a positive integer such that for each rational prime $p$ we have
    	\begin{equation*}
    		\#\{\Sigma_{p}\text{ is a }G-\text{specification at }p\}\leq a.
    	\end{equation*}
    	Define $T:=T_{\mathbb{Q}}(\lvert G\rvert d_{K})$ and recall $I_{T}^{+,\mu}$ from Definition~\ref{def: the set of G-specifications Phi}.
    	For each $\gamma>0$, define
    	\begin{equation*}
    		D_{\gamma}(s):=
    		a^{\omega(\bar{C}(K))}
    		\prod_{\mathfrak{p}\mid\lvert G\rvert\infty}\Bigl(\sum_{\chi_{\mathfrak{p}}:K_{\mathfrak{p}}^{*}\to N}\mathfrak{Nd}_{\chi_{\mathfrak{p}}}^{-s}\Bigr)
    		\cdot\sum_{
    			\substack{
    				d\in I_{T}^{+,\mu}\\
    				\omega(d)=\gamma
    			}
    		}\prod_{p\mid d}ap^{-s}=\sum_{n=1}^{\infty}a_{\gamma,n}n^{-s}.
    	\end{equation*}
    	Recall that $C(L)=\bar{C}(K)C'(L/K)$.
    	This implies that for each positive $n$, the set $\{L\in\mathcal{C}_{\Omega}^{\gamma}\mid K\subseteq L,\, C'(L/K)=n\}$ is nonempty only if $n$ satisfies the following two conditions simultaneously:
    	\begin{enumerate}
    		\item for each $p\in T$, if $p\mid\bar{C}(K)$ then $p\nmid n$;
    		\item there exists $d\in I_{T}^{+,\mu}$ with $\omega(d)=\gamma$ such that
    		\begin{equation*}
    			d\mid n.
    		\end{equation*} 
    	\end{enumerate}
    	Let $n$ be an integer with the above property.
    	Apply the Lemma~\ref{lemma: restricted ramification for abelian extensions} for the $H$-morphisms $\operatorname{C}_{K}\to N$, we have
    	\begin{align*}
    		& \#\{L\in\mathcal{C}_{\Omega}^{\gamma}\mid K\subseteq L,\, C'(L/K)=n\}\\
    		\leq & \lvert\operatorname{Hom}_{H}(\operatorname{Cl}_{K},N)\rvert
    		\cdot\prod_{\mathfrak{p}\mid\gcd(n,\lvert G\rvert)}\lvert\operatorname{Hom}(K_{\mathfrak{p}}^{*},N)\rvert \\
    		&\cdot\prod_{p\mid n\bar{C}(K),p\in T}\#\{\Sigma_{p}\text{ is a }G\text{-specification at }p\}\\
    		\leq & \lvert\operatorname{Hom}_{H}(\operatorname{Cl}_{K},N)\rvert a_{\gamma,n}.
    	\end{align*}
    	By the author~\cite[the Lemma 5.1 and 5.2]{wang2026invariant}, there exists some constant $c_{\gamma}$ such that for all $X>0$, we have
    	\begin{equation*}
    		\sum_{n<X}a_{\gamma,n}\leq c_{\gamma}\max\{1,\frac{X}{\log X}(\log\log X)^{\gamma-1}\}.
    	\end{equation*}
    	So, for each $K\in\mathcal{D}$, we have
    	\begin{align*}
    		& \#\{L\in\mathcal{C}_{\Omega}^{\gamma}\mid K\subseteq L,\, C'(L/K)<X\} \\
    		\leq & \lvert\operatorname{Hom}_{H}(\operatorname{Cl}_{K},N)\rvert\sum_{n<X}a_{\gamma,n} \\
    		\leq & c_{\gamma}\lvert\operatorname{Hom}_{H}(\operatorname{Cl}_{K},N)\rvert\max\{1,\frac{X}{\log X}(\log\log X)^{\gamma-1}\}.
    	\end{align*}
    \end{proof}
    The key point of this result is that the choice of the constant $c_{\gamma}$ is independent of the field $K\in\mathcal{D}$ as in the statement.
    Now we could prove the following.
    \begin{theorem}\label{thm: upper bound estimate S6}
    	Keep the notations in the Definition~\ref{def: notations of S6}.
    	If they satisfy the Hypothesis~\ref{hypthesis}, then for the prime $p\mid\lvert N\rvert$, for each positive integer $\gamma$, we have
    	\begin{equation*}
    		N_{\mathcal{C}_{\Omega}^{\gamma},C}(X)\ll X(\log X)^{b}(\log\log X)^{\gamma}.
    	\end{equation*}
    \end{theorem}
    \begin{proof}
    	We are going to do the computation for $N_{\mathcal{C}_{\Omega}^{\gamma},C}(X)$.
    	Recall that there exists some non-negative integer $b$ such that
    	\begin{equation*}
    		h(X):=\sum_{
    			\substack{
    				K\in\mathcal{D}\\
    				\bar{C}(K)<X
    			}
    		}\lvert\operatorname{Hom}_{H}(\operatorname{Cl}_{K},N)\rvert\ll X(\log X)^{b}
    	\end{equation*}
    	as $X\to\infty$.
    	For simplicity, define 
    	\begin{equation*}
    		f_{\gamma}(x):=\left\{\begin{aligned}
    			&1& &\text{ if }& x\in(0,e] \\
    			&\max\{1,\frac{x}{\log x}(\log\log x)^{\gamma-1}\}& &\text{ else if }& x>e,
    		\end{aligned}\right.
    	\end{equation*}
    	and
    	\begin{equation*}
    		g(x):=\left\{\begin{aligned}
    			&1& &\text{ if }& x\in(0,1] \\
    			&\max\{1,x(\log x)^{b}\}& &\text{ else if }& x>1.
    		\end{aligned}\right.
    	\end{equation*}
    	We first have the following estimate
    	\begin{align*}
    		N_{\mathcal{C}_{\Omega}^{\gamma},C}(X)
    		=&\sum_{
    			\substack{
    				K\in\mathcal{D}\\
    				\bar{C}(K)<X
    			}
    		}\#\{L\in\mathcal{C}_{\Omega}^{\gamma}\mid K\subseteq L,\, C'(L/K)<\frac{X}{\bar{C}(K)}\} \\
    		\leq & \sum_{
    			\substack{
    				K\in\mathcal{D}\\
    				\bar{C}(K)<X
    			}
    		}c_{\gamma}\lvert\operatorname{Hom}_{H}(\operatorname{Cl}_{K},N)\rvert f_{\gamma}(X/\bar{C}(K)) \\
    		=&c_{\gamma}\int_{1}^{X}f_{\gamma}(X/t)\operatorname{d}h(t)
    	\end{align*}
    	where the integral of the last row is the Riemann-Stieltjes integral.
    	It suffices to consider the integral.  
    	Since $f_{\gamma}(x)$ is a continuous function, and $h(x)$ is a step function, we know that the integral-by-parts holds in this case.
    	To be precise, by Hugh L. Montgomery and Robert C. Vaughan~\cite[Theorem A.1 and A.2]{montgomery2006multiplicative}, we first know that $\int_{1}^{X}f_{\gamma}(X/t)\operatorname{d}h(t)$ exists, hence
    	\begin{align*}
    		I:=\int_{1}^{X}f_{\gamma}(X/t)\operatorname{d}h(t)
    		=& f_{\gamma}(X/t)h(t)\Big\vert_{1}^{X}-\int_{1}^{X}h(t)\operatorname{d}f_{\gamma}(X/t) \\
    		(u=X/t)\Rightarrow=& h(X)-f_{\gamma}(X)h(1)+\int_{1}^{X}h(X/u)\operatorname{d}f_{\gamma}(u)   		
    	\end{align*}
    	Note that for all positive integer $\gamma$, the function $f_{\gamma}(x)$ is increasing when $x$ is large enough.
    	So, as $X\to\infty$, we have
    	\begin{align*}
    		I\ll & g(X)+\int_{1}^{X}g(X/u)\operatorname{d}f_{\gamma}(u) \\
    		\leq & g(X)+\int_{1}^{X}\frac{X}{u}(\log X)^{b}\operatorname{d}f_{\gamma}(u) \\
    		= & X(\log X)^{b}+\frac{X(\log X)^{b}}{u}f_{\gamma}(u)\Big\vert_{1}^{X}-(\log X)^{b}\int_{1}^{X}f_{\gamma}(u)\operatorname{d}\frac{X}{u} \\
    		\ll & X(\log X)^{b}+X(\log X)^{b-1}(\log\log X)^{\gamma-1}+X(\log X)^{b}\int_{e}^{X}\frac{(\log\log u)^{\gamma-1}}{u\log u}\operatorname{d}u \\
    		=& X(\log X)^{b-1}(\log\log X)^{\gamma-1}+X(\log X)^{b}+\frac{X(\log X)^{b}}{\gamma}(\log\log u)^{\gamma}\Big\vert_{e}^{X} \\
    		=& (\gamma^{-1}+o(1))X(\log X)^{b}(\log\log X)^{\gamma}.
    	\end{align*}    	
    	Therefore, for each positive integer $\gamma$, we have shown that 
    	\begin{equation*}
    		N_{\mathcal{C}_{\Omega}^{\gamma},C}(X)\ll X(\log X)^{b}(\log\log X)^{\gamma}.
    	\end{equation*}
    \end{proof}
    The above statement is an upper bound estimate for the field-counting $N_{\mathcal{C}_{\Omega}^{\gamma},C}(X)$.
    Using this information, we could show the following result on the statistics of the class groups.
    \begin{theorem}\label{thm: infinite moment S6}
    	Keep the notations as in the Definition~\ref{def: notations of S6}.
    	If they satisfy the Hypothesis~\ref{hypthesis}, and further more
    	\begin{equation*}
    		N_{\mathcal{C},C}(X)\gg X(\log X)^{b+1},
    	\end{equation*}
    	then for each non-negative integer $r$ we have that
    	\begin{equation*}
    		\mathbb{P}_{\mathcal{C},C}(\operatorname{rk}_{p}\operatorname{Cl}_{L}\leq r)
    		\quad\text{and}\quad
    		\mathbb{E}_{\mathcal{C},C}(\lvert\operatorname{Hom}(\operatorname{Cl}_{L},\mathbb{Z}/p\mathbb{Z})\rvert)=+\infty
    	\end{equation*}
    	where $p$ is the prime dividing $\lvert N\rvert$.
    \end{theorem}
    \begin{proof}
    	Claim: for each non-negative integer $\gamma$, we have
    	\begin{equation*}
    		N_{\mathcal{C}_{\Omega}^{\gamma},C}(X)=o(N_{\mathcal{C},C}(X)).
    	\end{equation*}
    	If the Claim is true, then the statement is just a direct corollary of the author~\cite[Theorem 3.14]{wang2026invariant}.
    	
    	By our condition and Theorem~\ref{thm: upper bound estimate S6}, this is true for each positive integer $\gamma$.
    	So, it suffices to prove the case when $\gamma=0$.
    	For each $K\in\mathcal{D}$, recall that the surjective continuous homomorphisms $\operatorname{C}_{K}\to N$ that is $H$-equivariant correspond to the Galois $N$-extensions $L/K$ such that $L/\mathbb{Q}$ is a Galois $G$-field.
    	This implies that there is a one-to-one correspondence between $\mathcal{C}_{\Omega}^{0}$ and $\bigcup_{K\in\mathcal{D}}\operatorname{Sur}_{H}(\operatorname{Cl}_{K},N)$.
    	Therefore, we have
    	\begin{align*}
    		N_{\mathcal{C}_{\Omega}^{0},C}(X)
    		=& \sum_{
    			\substack{
    				K\in\mathcal{D} \\
    				\bar{C}(K)<X
    			}
    		}\#\{L\in\mathcal{C}_{\Omega}^{0}\mid K\subseteq L,\, C(L)<X\} \\
    		\leq & \sum_{
    			\substack{
    				K\in\mathcal{D} \\
    				\bar{C}(K)<X
    			}
    		}\lvert\operatorname{Sur}_{H}(\operatorname{Cl}_{K},N)\rvert \\
    		\leq & \sum_{
    			\substack{
    				K\in\mathcal{D} \\
    				\bar{C}(K)<X
    			}
    		} \lvert\operatorname{Hom}_{H}(\operatorname{Cl}_{K},N)\rvert\ll X(\log X)^{b}.
    	\end{align*}
    	This proves that the Claim is true when $\gamma=0$.
    	So the Claim is true, hence also the theorem.
    \end{proof}
    \begin{remark}
    	The condition of the Theorem~\ref{thm: infinite moment S6} could be explained as a combination of the Cohen-Lenstra-Martinet Heuristics (see Cohen and Lenstra~\cite{cohen1984heuristics}, and Martinet and Cohen~\cite{martinet1990etude}) and the Malle-Bhargava Heuristics (see Malle~\cite{malle2004distribution}, and Bhargava~\cite{bhargava07mass}) in weak forms.
    	To be precise, for fields $K\in\mathcal{D}$, if the $N$-moment exists, then its main term should be the same as the one of the field-counting up to different coefficients, when the prime $p$ is a good prime.
    	Then by the heuristics on the field counting, $N_{\mathcal{C},C}(X)$ in general should admit a main term ``larger'' than $N_{\mathcal{D},C}(X)$, that is,
    	\begin{equation*}
    		N_{\mathcal{D},C}(X)=o(N_{\mathcal{C},C}(X)),
    	\end{equation*} 
    	hence motivating the assumption of the theorem.
    \end{remark}
    For the statement~\ref{thm: cubic fields S1} on the $S_{3}$-field in the Section~\ref{section: intro}, it could be taken as a special case of the Theorem~\ref{thm: infinite moment S6}.
    \begin{proof}[Proof of Theorem~\ref{thm: cubic fields S1}]
    	Note that $S_{3}$ admits the short exact sequence
    	\begin{equation*}
    		1\to\mathbb{Z}/3\mathbb{Z}\to S_{3}\to\mathbb{Z}/2\mathbb{Z}\to1
    	\end{equation*} 
    	So, let $\Omega:=\{(123),(132)\}$.
    	The set $\mathcal{C}$ of non-Galois cubic fields and the set of Galois $S_{3}$-fields are in the one-to-one correspondence  to each other.
    	In particular, we could define $\mathcal{C}_{\Omega}^{\gamma}$ via this correspondence.
    	And it admits a simple interpretation: a non-Galois cubic field $K\in\mathcal{C}_{\Omega}^{\gamma}$ if and only if it admits exactly $\gamma$ totally ramified primes other than $2$ and $3$.
    	By Roquette and Zassenhaus~\cite[Theorem 1]{RZ1969ClassRank}, for a cubic field $K$ with its Galois closure $L$, we have that
    	\begin{align*}
    		\operatorname{rk}_{3}\operatorname{Cl}_{L}\geq & \operatorname{rk}_{3}\operatorname{Cl}_{K} \\
    		\geq & \#\{p\nmid\infty:p\text{ is totally ramified in }K/\mathbb{Q}\} \\
    		\geq & \#\{p\nmid 6\infty:\rho_{p}(y_{p}^{t})\in\Omega\}-4,
    	\end{align*}
    	where $\rho_{p}:G_{\mathbb{Q}_{p}}\to S_{3}$ is the map corresponding to $L_{p}$.    	
    	We may also order $\mathcal{C}$ by the product of ramified primes $C$ of the Galois $S_{3}$-fields (the Definition~\ref{def: notations of S6}), that is, $C(K):=C(L)$ for a non-Galois cubic field $K$ with its Galois closure $L$.
    	By Shankar and Thorne~\cite{shankar2024asymptoticscubicfieldsordered}, we know that
    	\begin{equation*}
    		N_{\mathcal{C},\sqrt{d_{K}}}(X)\asymp X\log X.
    	\end{equation*}
    	Though our definition of $C(K)$ is different from the radical of the discriminant $\sqrt{d_{K}}$, their result allows a finite collection of local specifications (see \cite[Theorem 22]{shankar2024asymptoticscubicfieldsordered}).
    	For example, counting real cubic fields unramified at $2$ and $3$ by the radical of the discriminant has the main term $X\log X$.
    	In this case, the radical of the discriminant is literally the same as the product of ramified primes defined in the Definition~\ref{def: notations of S6}.
    	This implies that
    	\begin{equation*}
    		N_{\mathcal{C},C}(X)\gg X\log X.
    	\end{equation*}
    	This shows that the set $\mathcal{C}$ of non-Galois cubic fields ordered by the product of ramified primes $C$ satisfies the condition of the Theorem~\ref{thm: infinite moment S6}.
    	So the Claim in the proof is true, that is, for each non-negative integer $\gamma$, we have that
    	\begin{equation*}
    		N_{\mathcal{C}_{\Omega}^{\gamma},C}(X)\ll N_{\mathcal{C},C}(X).
    	\end{equation*}
    	Moreover, let $T:=T(6)=\{p\in\mathcal{P}\mid p\nmid 6\infty\}$.
    	Then the tuple $((\mathcal{C},C),\Omega,T)$ satisfies the condition of the author~\cite[Theorem 3.14]{wang2026invariant}.
    	So, the theorem is true.
    \end{proof}

\bibliographystyle{plain}
\bibliography{references}
\end{document}